\documentclass[10pt]{amsart}
\usepackage{amssymb}
\usepackage{bm}
\usepackage{graphicx}
\usepackage[centertags]{amsmath}
\usepackage{amsfonts}
\usepackage{amsthm}
\usepackage{graphicx}
\newtheorem{thm}{Theorem}
\newtheorem{cor}[thm]{Corollary}
\newtheorem{lem}[thm]{Lemma}
\newtheorem{prop}[thm]{Proposition}
\newtheorem{defn}[thm]{Definition}
\theoremstyle{definition}
\newtheorem{rem}{Remark}

\newtheorem{examp}{Example}

\newcommand{\rr}{\mathbb{R}}

\newcommand{\ee}{\varepsilon}
\newcommand{\nn}{\mathbb{N}}
\newcommand{\ttt}{\mathcal{T}}

\newcommand{\ff}{\mathcal{F}}

\newcommand{\g}{\mathcal{G}}

\numberwithin{thm}{section}

\begin{document}

\title[Higher Order Spreading Models]{Higher Order Spreading Models}
\author{S. A. Argyros}
\address{National Technical University of Athens, Faculty of Applied Sciences,
Department of Mathematics, Zografou Campus, 157 80, Athens, Greece}
\email{sargyros@math.ntua.gr}
\author{V. Kanellopoulos}
\address{National Technical University of Athens, Faculty of Applied Sciences,
Department of Mathematics, Zografou Campus, 157 80, Athens, Greece}
\email{bkanel@math.ntua.gr}
\author{K. Tyros}
\address{Department of
Mathematics, University of Toronto, Toronto, Canada, M5S 2E4}
\email{ktyros@math.toronto.edu}

\maketitle

\begin{abstract} We introduce the higher order spreading models
associated to a Banach space $X$. Their definition is based on
$\ff$-sequences $(x_s)_{s\in\ff}$ with $\ff$  a regular thin family
and the plegma families. We show that the higher order spreading
models of a Banach space $X$ form an increasing  transfinite
hierarchy $(\mathcal{SM}_\xi(X))_{\xi<\omega_1}$. Each
$\mathcal{SM}_\xi (X)$ contains all spreading models generated by
$\ff$-sequences $(x_s)_{s\in\ff}$ with order of $\ff$ equal to
$\xi$. We also provide a study of the fundamental properties of the
hierarchy.
\end{abstract}



%\tableofcontents
\section{Introduction} Spreading models have been invented by A.
Brunel and L. Sucheston \cite{BS} in the middle of 70's   and since
then  have a constant presence in the evolution of the Banach space
theory. Recall that a sequence $(e_n)_{n}$ in a seminormed space
$(E,\|\cdot\|_*)$ is called a  spreading model of the space $X$ if
there exists a sequence $(x_n)_{n}$ in $X$ which is \textit{Schreier
almost isometric} to $(e_n)_{n}$, that is for some
 null sequence $(\delta_n)_{n}$ of positive reals
we have
 \begin{equation}\label{wsp}\Bigg|\Big\|\sum_{i=1}^ka_ix_{n_i}\Big\|
 -\Big\|\sum_{i=1}^ka_ie_i\Big\|_*\Bigg|<\delta_{k},\end{equation}
for every $k \leq n_1<\ldots<n_k$  and $(a_i)_{i=1}^k\in[-1,1]^k$.
We also say that the sequence $(x_n)_n$ which satisfies (\ref{wsp})
generates $(e_n)_n$ as a spreading model. By an iterated use of
Ramsey's theorem  \cite{R}, Brunel and Sucheston proved that every
bounded sequence  in a Banach space $X$ has a subsequence generating
a spreading model.

 It is easy to see that any sequence $(e_n)_{n}$
satisfying (\ref{wsp})  is  spreading\footnote{A sequence
$(e_n)_{n}$ in a seminormed space $(E,\|\cdot\|_*)$ is called
spreading if for every $n\in\nn$, $k_1<\ldots<k_n$ in $\nn$ and
$a_1,\ldots,a_n\in\rr$ we have that $\|\sum_{j=1}^na_j
e_j\|_*=\|\sum_{j=1}^n a_j e_{k_j}\|_*$}. The importance of the
spreading models arise from the fact that they connect in an
asymptotic manner the structure of an arbitrary Banach space $X$ to
the corresponding one of spaces generated by spreading sequences.
 The definition of the
spreading model resembles the  finite representability\footnote{A
Banach space $Y$ is finitely representable in $X$ if for every
finite dimensional subspace $F$ of $Y$ and every $\ee>0$ there
exists $T:F\to Y$ bounded linear injection such that
$\|T\|\cdot\|T^{-1}\|<1+\ee$.} of the space generated by the
sequence $(e_n)_{n}$ into the space $(X,\|\cdot\|)$. However there
exists a significant difference between the two concepts. Indeed in
the frame of the finite representability there are two classical
achievements: Dvoretsky's theorem  \cite{D} asserting that $\ell^2$
is finitely representable in every Banach space $X$ and also
Krivine's theorem \cite{Kr} asserting that for every linearly
independent sequence $(x_n)_{n}$ in $X$ there exists a $1\leq p\leq
\infty$ such that $\ell^p$ is block finitely representable in the
subspace generated by $(x_n)_n$. On the other hand E. Odell and Th.
Schlumprecht \cite{O-S} have shown that there exists a reflexive
space $X$ admitting no $\ell^p$ as a spreading model. Thus the
spreading models of a space $X$ lie strictly between the finitely
representable spaces in $X$ and the spaces that are isomorphic to a
subspace of $X$.

The spreading models associated to a Banach space $X$ can be
considered as a cloud of Banach spaces, including many members
with regular structure, surrounding the space $X$ and offering
information concerning the local structure of $X$ in an asymptotic
manner. Our aim is to enlarge that cloud and to fill in the gap
between spreading models and the spaces which are finitely
representable in $X$. More precisely we extend the
Brunel--Sucheston concept of a spreading model and we show that
under the new definition the spreading models associated to a
Banach space $X$ form a whole hierarchy of classes of spaces
indexed by the countable ordinals. The first class of this
hierarchy is the classical spreading models. The initial step of
this extension  has already been done in \cite{AKT} where  the
class of $k$-spreading models was defined for every positive
integer $k$. The transfinite extension   introduced in the present
paper requires analogue ingredients that we are about to describe.

 The first one is the
 $\ff$-sequences; that is sequences
 of the form  $(x_s)_{s\in\ff}$ where the index set $\ff$ is a regular thin family  of finite subsets of $\nn$
 (see Definition \ref{defn regular thin}). Typical examples of
 such families are the $k$-element subsets of $\nn$  and also
 the maximal elements of the $\xi^{th}$-Schreier family $\mathcal{S}_\xi$ (see \cite{AA}).
 A subsequence of
$(x_s)_{s\in\ff}$  is a restriction of the $\ff$-sequence on an
infinite subset of $\nn$, i.e. it is of the form
$(x_s)_{s\in\ff\upharpoonright M}$ where $\ff\upharpoonright
M=\ff\cap [M]^{<\infty}$. Among others we study the convergence of
$\ff$-sequences in a topological space $(X,\mathcal{T})$. In this
setting we show that when the closure of $(x_s)_{s\in\ff}$ in
$(X,\mathcal{T})$ is a compact metrizable space then we can always
restrict to an infinite subset $M$  of $\nn$ where the subsequence
$(x_s)_{s\in\ff\upharpoonright M}$ is \textit{subordinated}; that is
if  $\widehat{\ff}=\{t\in [\nn]^{<\infty}: \exists s\in \ff  \
\text{such that} \ t \ \text{is an initial segment of} \ s\}$  then
there exists a continuous map
$\varphi:\widehat{\mathcal{F}}\upharpoonright M\to X$ with
$\varphi(s)=x_s$ for every $s\in\ff\upharpoonright M$ (see
Definition \ref{defsub} and Theorem \ref{Create subordinated}).

 The second ingredient   is the \textit{plegma families}
which extend the corresponding one in  \cite{AKT}. Roughly
speaking a plegma family is a sequence $(s_1,\ldots,s_l)$ of non
empty finite subsets of $\nn$ where the first elements  of
$s_1,...,s_l$ are in increasing order and they lie before their
second elements which are also in increasing order and so on (see
Definition \ref{defn plegma}). Here the plegma families do not
necessarily include sets of equal size.

The $\ff$-sequences and the plegma families are the key components
for the definition of the higher order spreading models which goes
as follows. Given an $\ff$-sequence $(x_s)_{s\in\ff}$ in a Banach
space $X$ and a sequence $(e_n)$ in a seminormed space
$(E,\|\cdot\|_*)$ we will say that $(x_s)_{s\in\ff}$
\textit{generates} $(e_n)_n$ \textit{as an}
$\ff$-\textit{spreading model} if for some
 null sequence $(\delta_n)_{n}$ of positive reals
we have
 \begin{equation}\label{wspp}\Bigg|\Big\|\sum_{i=1}^ka_ix_{s_i}\Big\|
 -\Big\|\sum_{i=1}^ka_ie_i\Big\|_*\Bigg|<\delta_{k},\end{equation}
for every $(a_i)_{i=1}^k\in[-1,1]^k$ and every plegma family
$(s_i)_{i=1}^k$ in $\ff$ with $k \leq \min s_1$. Note  that the
$\ff$-sequences $(x_s)_{s\in\ff}$  generate a higher order spreading
model just as the ordinary sequences $(x_n)_n$ do in the classical
definition. Moreover since the family of all $k$-element subsets of
$\nn$ is a regular thin family, the above definition extends   the
classical definition of the spreading model as well as   the one of
$k$-spreading models given in \cite{AKT}.

Brunel--Sucheston's theorem \cite{BL} is extended in the frame of
the bounded $\ff$-sequences $(x_s)_{s\in\ff}$ in a Banach space
$X$. Namely,  every bounded $\ff$-sequence  in $X$  contains a
subsequence $(x_s)_{s\in\ff\upharpoonright M}$ generating an
$\ff$-spreading model. The proof is based on the fact that plegma
families with elements in a regular thin family satisfy strong
Ramsey properties. It is notable that the concept of the
$\ff$-spreading model is independent of the particular family
$\ff$ and actually depends only on the order\footnote{The
\textit{order} of $\ff$, denoted by $o(\ff)$ is a countable
ordinal which measures the complexity $\ff$ (see Section 2 for the
precise definition).  For example the family of $k$-element
subsets of $\nn$ has order $k$, while the $\xi^{th}$-Schreier has
order $o(\mathcal{S}_\xi)=\omega^\xi$} of the family $\ff$.
Namely, if $(e_n)_n$ is an $\ff$-spreading model then it is also a
$\g$-spreading model for every regular thin family $\g$ with
$o(\g)\geq o(\ff)$. This fact allow us to classify all the
$\ff$-spreading models of a Banach space $X$ as   an increasing
transfinite hierarchy of the form
$(\mathcal{SM}_\xi(X))_{\xi<\omega_1}$. Let's point out that the
$\xi$-spreading models of $X$ have a weaker asymptotic relation to
the space $X$ as $\xi$ increases to $\omega_1$.

The infinite graphs with vertices from a regular thin family and
edges the plegma pairs is the  key for   the proof of  the above
results. Specifically, it is  shown that if $\g$ and $\ff$ are two
regular thin families with  $o(\g)\geq o(\ff)$ then there exist an
infinite subset $M$ of $\nn$ and a \text{plegma preserving} map
$\varphi:\g\upharpoonright M\to \ff$ (that is
$(\varphi(s_1),\varphi(s_2))$ is a plegma pair in $\ff$ whenever
$(s_1,s_2)$ is plegma pair in $\g\upharpoonright M$). Moreover, it
is also shown that such an embedding is forbidden if we wish to go
from families of lower to families of higher order. More
precisely, if $o(\ff)<o(\g)$ then for every $M\in[\nn]^\infty$ and
$\varphi:\ff\upharpoonright M\to\g$ there exists $L\in[M]^\infty$
such that for every plegma pair $(s_1,s_2)$ in $\ff\upharpoonright
L$ neither $(\phi(s_1),\phi(s_2))$ nor $(\phi(s_2),\phi(s_1))$ is
a plegma pair in $\g$ (see Theorems \ref{plegma preserving maps}
and \ref{non plegma preserving maps}).

The paper is organized as follows. In Section 2 we review some basic
facts concerning families of finite subsets of $\nn$ and we define
the regular thin families. In Section 3 we study the plegma families
and their properties. In Section 4 we introduce the definition of
the higher order spreading models. In Section 5 we deal with
$\ff$-sequences $(x_s)_{s\in\ff}$  in a general topological space.
Finally, in Section 6 we study the $\ff$-sequences which generate
several classes of spreading sequences as  spreading models. In this
last section we show that several well known results concerning the
classical spreading models remain valid in the higher order setting.
For instance, we show that a subordinated, seminormalized and weakly
null $\ff$-sequence generates an unconditional spreading model.

The present paper is an updated version of the first part of
\cite{AKT1}. The second part which deals with certain examples
will be presented elsewhere.

\subsection{Preliminary notation and definitions} By
$\nn=\{1,2,\ldots\}$ we denote the set of all positive integers.
Throughout the paper we shall identify strictly increasing sequences
in $\nn$ with their corresponding range i.e. we view every strictly
increasing sequence in $\nn$ as a subset of $\nn$ and conversely
every subset of $\nn$ as the sequence resulting from the increasing
ordering of its elements. We will use capital letters as $L,M,N,...$
to denote infinite subsets  and lower case letters as $s,t,u,...$ to
denote finite subsets of $\nn$. For every infinite subset $L$ of
$\nn$, $[L]^{<\infty}$ (resp. $[L]^\infty$) stands for the set of
all finite (resp. infinite) subsets of $L$. For an
$L=\{l_1<l_2<...\}\in [\nn]^\infty$ and a positive integer
$k\in\nn$, we set $L(k)=l_k$. Similarly, for a finite subset
$s=\{n_1<..<n_m\}$ of $\nn$ and  for $1\leq k\leq m$ we set
$s(k)=n_k$. For an $L=\{l_1<l_2<...\}\in [N]^\infty$ and a finite
subset $s=\{n_1<..<n_m\}$ (resp. for an infinite subset
$N=\{n_1<n_2<...\}$ of $\nn$), we set
$L(s)=\{l_{n_1},...,l_{n_m}\}=\{L(s(1)),...,L(s(m))\}$ (resp.
$L(N)=\{l_{n_1},l_{n_2},...\}=\{L(N(1)), L(N(2)),...\}$).

For $s\in[\nn]^{<\infty}$ by $|s|$ we denote the cardinality of $s$.
For $L\in[\nn]^\infty$ and $m\in\nn$, we denote by $[L]^m$ the set
of all $s\in[L]^{<\infty}$ with $|s|=m$. Also for every nonempty
$s\in[\nn]^{<\infty}$ and $1\leq k\leq |s|$ we set
$s|k=\{s(1),\ldots,s(k)\}$ and $s|0=\varnothing$. Moreover, for
$s,t\in[\nn]^{<\infty}$, we write $t\sqsubseteq s$ (resp.
$t\sqsubset s$) to denote that $t$ is an initial (resp.
\textit{proper} initial) segment of $s$. Also, for every
$s,t\in[\nn]^{<\infty}$ we write $t<s$ if either at least one of
them is the empty set, or $\max t<\min s$.

Concerning Banach space theory, although the notation that we follow
is the standard one,  we present for the sake of completeness some
basic concepts that we will need. Let $X$ be a Banach space.
 We say that a sequence $(x_n)_{n}$ in $X$ is bounded
(resp. seminormalized) if there exists $M>0$ (resp. $M_1,M_2>0)$
such that $\|x_n\|\leq M$ (resp. $M_1\leq \|x_n\|\leq M_2$) for all
$n\in\nn$. The sequence $(x_n)_n$ is  called Schauder basic if there
exists a constant $C\geq 1$ such that
\begin{equation}\Big\|\sum_{n=1}^k a_nx_n\Big\|\leq
C\Big\|\sum_{n=1}^m a_nx_n\Big\|,\end{equation} for every $
k\leq m$ in $\nn$ and $a_1,...,a_m\in\rr$. Finally, we say that
$(x_n)_{n}$ is  $C$-unconditional,  if for every $m\in\nn$,
$F\subseteq \{1,...,m\}$ and $a_1,...,a_m\in\rr$, it holds that
\begin{equation}\Big\|\sum_{n\in F} a_nx_n\Big\|\leq
C\Big\|\sum_{n=1}^m
a_nx_n\Big\|.\end{equation}

\section{Regular thin families}
In this section we define  the regular thin families of finite
subsets of $\nn$ and we study their basic properties. The
definition is based on two well known concepts, namely that of
regular families traced back to \cite{AA} and that of thin
families defined in \cite{N-W} and extensively studied in
\cite{Pd-R} and \cite{LT}.
\subsection{On families of finite subsets of $\nn$}
We start with  a review of the basic concepts concerning families
of finite subsets of $\nn$. For a more detailed exposition the
reader can refer to \cite{AT}.
\subsubsection{Ramsey properties of families of finite subsets of
$\nn$} For a family  $\ff\subseteq [\nn]^{<\infty}$ and $L\in
[\nn]^\infty$, we set \begin{equation}\ff\upharpoonright L=\{s\in
\ff:\;s\subseteq L\}=\ff\cap [L]^{<\infty}.\end{equation}

 Recall
the following terminology from \cite{G3}. Let
$\ff\subseteq[\nn]^{<\infty}$ and $M\in[\nn]^\infty$. We say that
$\ff$ is \textit{large} in $M$ if for every $L\in[M]^\infty$,
$\ff\upharpoonright L$ is nonempty. We say that $\ff$ is
\textit{very large} in $M$ if for every $L\in[M]^\infty$ there
exists $s\in\ff$ such that $s\sqsubseteq L$. The following is a
restatement  ( see \cite{G3}) of a well known theorem of C. St. J.
A. Nash-Williams \cite{N-W} and F. Galvin and K. Prikry \cite{G-P}.
\begin{thm} \label{Galvin prikry}
Let $\ff\subseteq[\nn]^{<\infty}$ and $M\in[\nn]^\infty$. If $\ff$
is large in $M$ then there exists $L\in[M]^\infty$ such that $\ff$
is very large in $L$.
\end{thm}

\subsubsection{The order of a family of finite subsets of $\nn$}
Let $\ff\subseteq [\nn]^{<\infty}$ be  a nonempty family of finite
subsets of $\nn$. The \textit{order} of $\ff\subseteq
[\nn]^{<\infty}$ is defined as follows (see also \cite{Pd-R}).
First, we assign to $\ff$ its ($\sqsubseteq-$)\textit{closure},
i.e. the set
  \begin{equation}\widehat{\ff}=\{t\in[\nn]^{<\infty}:\exists  s\in\ff\text{ with }t\sqsubseteq
  s\},\end{equation}
which is a tree under the initial segment ordering. If
$\widehat{\ff}$ is ill-founded (i.e. there exists an infinite
sequence $(s_n)_{n}$ in $\widehat{\ff}$ such that $s_n\sqsubset
s_{n+1}$) then we set $o(\ff)=\omega_1$. Otherwise for every maximal
element $s$ of $\widehat{\ff}$ we set $o_{\widehat{\ff}}(s)=0$ and
recursively for every $s$ in $\widehat{\ff}$ we define
\begin{equation}
o_{\widehat{\ff}}(s)=\sup\{o_{\widehat{\ff}}(t)+1:t\in\widehat{\ff}\text{
and }s\sqsubset t\}.\end{equation} The order of $\ff$ denoted by
$o(\ff)$ is defined to be the ordinal
$o_{\widehat{\ff}}(\varnothing)$. For instance
$o(\{\varnothing\})=0$ and $o([\nn]^k)=k$, for every $k\in\nn$.

For every $n\in \nn$, we define
\begin{equation}\ff_{(n)}=\{s\in[\nn]^{<\infty}:n<s\text{ and }\{n\}\cup
s\in\ff\},\end{equation} where $n<s$ means that either $s=\varnothing$
or $n<\min s$. It is easy to see that for every nonempty family
$\ff\subseteq [\nn]^{<\infty}$ we have that
\begin{equation}\label{eq1}
  o(\ff)=\sup\{o(\ff_{(n)})+1:n\in\nn\}.
\end{equation}

\subsubsection{Regular families} A family $\mathcal{R}\subseteq
[\nn]^{<\infty}$  is said to be \textit{hereditary} if for every
$s\in\ff$ and $t\subseteq s$ we have that  $t\in\ff$ and
\textit{spreading} if for every $n_1<...<n_k$ and  $m_1<...<m_k$
with  $n_1\leq m_1$, ..., $n_k\leq m_k$, we have that
$\{m_1,...,m_k\}\in \mathcal{R}$ whenever
$\{n_1,...,n_k\}\in\mathcal{R}$. Also $\mathcal{R}$ is called
\textit{compact} if the set $\{\chi_s\in\{0,1\}^\nn:\;s\in
\mathcal{R}\}$ of characteristic functions of the members of
$\mathcal{R}$, is a closed subset of $\{0,1\}^\nn$ under the product
topology.

A family $\mathcal{R}$ of finite subsets of $\nn$ will be called
regular if it is compact, hereditary and spreading. Notice that for
every regular family $\mathcal{R}$,
$\widehat{\mathcal{R}}=\mathcal{R}$ and $\mathcal{R}_{(n)}$ is also
regular for every $n\in\nn$. Moreover, using equation (\ref{eq1})
and by induction on the order of $\mathcal{R}$  we easily get the
following.
\begin{prop}\label{orderreg}
Let $\mathcal{R}$ be a regular family. Then
$o(\mathcal{R}\upharpoonright L)=o(\mathcal{R})$, for every $L\in
[\nn]^\infty$.
\end{prop}

Exploiting the method of  \cite{Pd-R}  we have the next result.
\begin{prop}
For every $\xi<\omega_1$ there exists a regular family
$\mathcal{R}_\xi$ with $o(\mathcal{R}_\xi)=\xi$. \end{prop}
\begin{proof} For $\xi=0$ we set
$\mathcal{R}_0=\{\varnothing\}$. We proceed by induction on
$\xi<\omega_1$. Assume that for some
 $\xi<\omega_1$ and for each $\zeta<\xi$ we have defined a regular
 family $\mathcal{R}_\zeta$ with $o(\mathcal{R}_\zeta)=\zeta$. If
 $\xi$ is a successor ordinal, i.e. $\xi=\zeta+1$, then we set
 \[\mathcal{R}_{\xi}=\Big\{\{n\}\cup
 s:\;n\in\nn,\;s\in\mathcal{R}_\zeta\;\text{and}\;n<
 s\Big\}.\]
 If $\xi$ is a limit ordinal, then we choose a strictly increasing
 sequence $(\zeta_n)_{n}$ such that $\zeta_n\to\xi$ and we
 set
 \[\mathcal{R}_\xi=\bigcup_{n}\Big\{s\in\mathcal{R}_{\zeta_n}:\min
 s\geq
 n\Big\}=\bigcup_{n}\mathcal{R}_{\zeta_n}\upharpoonright[n,+\infty).\]
It is easy to check  that $\mathcal{R}_\xi$ is a regular family with
$o(\mathcal{R}_\xi)=\xi$ for all $\xi<\omega_1$.
 \end{proof}

We will need  some combinatorial properties of regular families. To
this end we give  the following definition.

For every $\mathcal{R}\subseteq [\nn]^{<\infty}$ and
$L\in[\nn]^\infty$, let
\begin{equation} L(\mathcal{R})=\{L(s):\;s\in \mathcal{R}\}.\end{equation}
Notice that $o(\mathcal{R})=o(L(\mathcal{R}))$ and if $\mathcal{R}$
is compact (resp. hereditary) then
 $L(\mathcal{R})$ is also compact (resp. hereditary).
It is also easily verified that if  $\mathcal{R}$ is  spreading then
$L_1(\mathcal{R})\subseteq L_2(\mathcal{R})$, for every
$L_1\subseteq L_2$ in $[\nn]^\infty$ and more generally,
\begin{equation}\label{eqlem}L_1(\mathcal{R}_{(k)})\subseteq L_2(\mathcal{R}_{(k)}),\end{equation} for
every $k\in \nn$ and $L_1,L_2\in[\nn]^\infty$ satisfying
$\{L_1(j):j> k\}\subseteq \{L_2(j):j> k\}$ (where
$L(\mathcal{R}_{(k)})=\{L(s):s\in\mathcal{R}_{(k)}\}$).

\begin{prop}\label{embending a family F in G by a shifting where o(F)<o(G)}
Let $\mathcal{R},\mathcal{S}$ be regular families of finite subsets
of $\nn$ with $o(\mathcal{R})\leq o(\mathcal{S})$. Then for every
$M\in[\nn]^\infty$ there exists $L\in [M]^\infty$ such that
$L(\mathcal{R})\subseteq\mathcal{S}$.
\end{prop}
\begin{proof}
If  $o(\mathcal{R})= 0$, i.e. $\mathcal{R}=\{\varnothing\}$, then
the conclusion trivially holds.  Suppose that for some
$\xi<\omega_1$ the proposition is true for every regular families
$\mathcal{R}',\mathcal{S}'\subseteq[\nn]^{<\infty}$ such that
$o(\mathcal{R}')<\xi$ and $o(\mathcal{R}')\leq o(\mathcal{S}')$. Let
$\mathcal{R}$, $\mathcal{S}$ be regular with $o(\mathcal{R})=\xi$
and let $M\in [\nn]^\infty$. By  equation (\ref{eq1}) we have that
$o(\mathcal{R}_{(1)})<o(\mathcal{R})$. Hence
$o(\mathcal{R}_{(1)})<o(\mathcal{S})$ and so there is some
$l_1\in\nn$ such that $o(\mathcal{R}_{(1)})\leq
o(\mathcal{S}_{(l_1)})$. Since  $\mathcal{S}$ is spreading we have
that $o(\mathcal{S}_{(l_1)})\leq o(\mathcal{S}_{(n)})$ for all
$n\geq l_1$ and therefore we may suppose that $l_1\in M$. Since
$\mathcal{R}_{(1)}$ and $\mathcal{S}_{(l_1)}$ are regular families,
by  our inductive hypothesis there is  $L_1\in [M]^\infty$ such that
$L_1(\mathcal{R}_{(1)})\subseteq \mathcal{S}_{(l_1)}$.

Proceeding in the same way we construct a strictly increasing
sequence $(l_j)_j$ in $M$ and a decreasing sequence $M=L_0\supset
L_1\supset ...$ of infinite subsets of $M$ such that
 (i) $l_{j+1}\in  L_j$, (ii) $l_{j+1}> L_j(j)$ and (iii)
$L_j(\mathcal{R}_{(j)})\subseteq \mathcal{S}_{(l_j)}$,
  for all $j\geq 1$.

We set  $L=\{l_j\}_{j}$ and we  claim that
$L(\mathcal{R})\subseteq\mathcal{S}$. Indeed, by the above
construction we have that for every $k\in\nn$,
$\{L(j)\}_{j>k}\subseteq \{ L_k(j)\}_{j> k}$. Therefore by
(\ref{eqlem}) and (iii) above,  we get that
\begin{equation}
\label{eq2}L(\mathcal{R}_{(k)})\subseteq
L_k(\mathcal{R}_{(k)})\subseteq \mathcal{S}_{(l_k)}.
\end{equation} It is easy to see
that $L(\mathcal{R}_{(k)})=L(\mathcal{R})_{(l_k)}$ and so by
(\ref{eq2}) we have that $L(\mathcal{R})_{(l_k)}\subseteq
\mathcal{S}_{(l_k)}$. Since this holds for every $k\in\nn$, we
conclude that $L(\mathcal{R})\subseteq \mathcal{S}$.
\end{proof}
The next corollary is an immediate consequence.
\begin{cor}
  Let $\mathcal{R},\mathcal{S}$ be regular families of finite subsets of $\nn$ with $o(\mathcal{R})= o(\mathcal{S})$.
  Then for every $M\in[\nn]^\infty$ there exists $L\in [M]^\infty$ such that $L(\mathcal{R})\subseteq\mathcal{S}$ and $L(\mathcal{S})\subseteq\mathcal{R}$.
\end{cor}
\subsubsection{Thin families} A family $\ff$ of finite subsets of
$\nn$ is called \textit{thin} if there do not exist  $s, t$ in
$\ff$ such that $s$ is a proper initial segment of $t$. The
following result is contained in \cite{N-W} and \cite{Pd-R}. Since
it plays a crucial role in the sequel for the sake of completeness
we present its proof.
\begin{prop}\label{th2}
Let $\ff\subseteq[\nn]^{<\infty}$ be a thin family. Then for every
finite partition $\ff=\cup_{i=1}^k \ff_i$, ($k\geq 2$) of $\ff$ and
every $M\in[\nn]^\infty$ there exist $L\in[M]^\infty$ and $1\leq
i_0\leq k$ such that $\ff\upharpoonright L\subseteq\ff_{i_0}$.
\end{prop}
\begin{proof} It suffices to  show the result only for  $k=2$ since the general case follows easily by induction.
 So let  $\ff=\ff_1\cup \ff_2$
and $M\in [\nn]^\infty$. Then either there is $L\in [M]^\infty$ such
that  $\ff_1\upharpoonright L=\varnothing$ or $\ff_1$ is large in
$M$. In the first case it is clear that  $\ff\upharpoonright
L\subseteq \ff_2$. In the second case by Theorem \ref{Galvin prikry}
there is $L\in [M]^\infty$ such that $\ff_1$ is very large in $L$.
We claim that $\ff\upharpoonright L\subseteq\ff_1$. Indeed, let
$s\in\ff\upharpoonright L$. We choose    $N\in [L]^\infty$ such that
$s\sqsubseteq N$ and let $t\sqsubseteq N$ such that $t\in \ff_1$.
Then $s,t$ are $\sqsubseteq$-comparable members of  $\ff$ and since
$\ff$ is thin $s=t\in\ff_1$. Therefore $\ff\upharpoonright
L\subseteq \ff_1$.
\end{proof}
\subsection{Regular thin families} We are now ready to introduce the
main concept of this section.
\begin{defn}\label{defn regular thin} A family
$\ff$ of finite subsets of $\nn$ will be called regular thin if (a)
$\ff$ is thin and (b) the $\sqsubseteq$-closure $\widehat{\ff}$ of
$\ff$ is a regular family.
\end{defn}
 The next lemma
allow us to construct regular thin families from regular ones. We
will use the following notation. For a  family $\mathcal{R}\subseteq
[\nn]^{<\infty}$ we set
\begin{equation}\mathcal{M}(\mathcal{R})=\{s\in\mathcal{R}:\;s\;\text{is}\;\sqsubseteq\text{-maximal
    in}\;\mathcal{R}\}.\end{equation}
Notice that a family  $\ff\subseteq [\nn]^{<\infty}$ is thin if and
only if $\ff=\mathcal{M(\widehat{\ff}})$.
\begin{lem}\label{maximal}
Let $\mathcal{R}$ be a regular family. Then the family
$\mathcal{M}(\mathcal{R})$ is  thin and satisfies
$\widehat{\mathcal{M}(\mathcal{R})}=\mathcal{R}$. Therefore
$\mathcal{M}(\mathcal{R})$ is regular thin with
$o(\mathcal{M}(\mathcal{R}))=o(\mathcal{R})$.
\end{lem}
\begin{proof}
Since $\mathcal{M}(\mathcal{R})\subseteq \mathcal{R}$ and
$\mathcal{R}$ is hereditary, we have that
$\widehat{\mathcal{M}(\mathcal{R})}\subseteq \mathcal{R}$. To show
that $\mathcal{R}\subseteq \widehat{\mathcal{M}(\mathcal{R})}$
notice that for every $s\in\mathcal{R}$ there exists a $t\in
\mathcal{M}(\mathcal{R})$ such that $s\sqsubseteq t$, otherwise
$\mathcal{R}$ would not be compact. Hence
$\widehat{\mathcal{M}(\mathcal{R})}=\mathcal{R}$ and  clearly
$\mathcal{M}(\mathcal{R})$ is thin. Thus $\mathcal{M}(\mathcal{R})$
is regular thin. Finally, by the definition of the order, we have
$o(\mathcal{M}(\mathcal{R}))=o(\widehat{\mathcal{M}(\mathcal{R})})$
and hence $o(\mathcal{M}(\mathcal{R}))=o(\mathcal{R})$.
\end{proof}
\begin{cor}
For  every $\xi<\omega_1$ there is a regular thin family $\ff_\xi$
with $o(\ff_\xi)=\xi$. \end{cor}
\begin{proof}
Let $\xi<\omega_1$ and  $\mathcal{R}_\xi$ be  a regular family with
$o(\mathcal{R}_\xi)=\xi$. Then
$\mathcal{F}_\xi=\mathcal{M}(\mathcal{R}_\xi)$ is as desired.
\end{proof}
\begin{cor}
The map which sends $\ff$ to $\widehat{\ff}$ is a bijection between
the set of all regular thin families and the set of all regular
ones. Moreover, the inverse map sends each regular family
$\mathcal{R}$ to $\mathcal{M}(\mathcal{R})$.
\end{cor}
\begin{proof}
  By the definition of regular thin families, the map $\ff\to\widehat{\ff}$ sends each
  regular thin family to a regular one. By Lemma
  \ref{maximal} we get that the map is 1-1,
   onto and the inverse
  map sends each regular
  family $\mathcal{R}$ to $\mathcal{M}(\mathcal{R})$.
\end{proof}

\begin{rem}
If $\ff$ is a regular thin family with $o(\ff)=k<\omega$, then it is
easy to see that there exists $n_0$ such that $\ff\upharpoonright
[n_0,\infty)=\{s\in[\nn]^k:\min s\geq n_0\}$. Therefore, for each
$k<\omega$, the family $[\nn]^k$ is essentially the unique regular
thin family of order $k$. However this does not remain valid for
regular thin families of order $\xi\geq \omega$.  For instance for
every unbounded increasing map $f:\nn\to\nn$ the family
$\ff=\{s\in[\nn]^{<\infty}: |s|=f(\min s)\}$, is a regular thin
family of order $\omega$.
\end{rem}
\begin{lem}\label{wx} Let $\mathcal{R}$ be a regular family and $L\in[\nn]^\infty$.
 Then
$\mathcal{M}(\mathcal{R})\upharpoonright
L=\mathcal{M}(\mathcal{R}\upharpoonright
 L)$ and setting  $\mathcal{M}=\mathcal{M}(\mathcal{R})$, $\widehat{\mathcal{M}\upharpoonright
L}=\mathcal{R}\upharpoonright
 L$ and $o(\mathcal{M}\upharpoonright L)=o(\mathcal{R})$.
\end{lem}
\begin{proof}
  It is easy to see that $\mathcal{M}(\mathcal{R})\upharpoonright L\subseteq \mathcal{M}(\mathcal{R}\upharpoonright
  L)$. To show the converse inclusion let $s\in \mathcal{M}(\mathcal{R}\upharpoonright
  L)$ and assume that  $s\notin \mathcal{M}(\mathcal{R})$. Since
  $\mathcal{R}\upharpoonright
  L\subseteq \mathcal{R}$,
  $s\in \mathcal{R}$ and therefore there exists some $t\in
  \mathcal{M}(\mathcal{R})$ with $s\sqsubset t$. Since
  $\mathcal{R}$ is spreading this yields that exists
  $t'\in\mathcal{R}\upharpoonright L$ with $s\sqsubset t'$. Thus
  $s\notin \mathcal{M}(\mathcal{R}\upharpoonright L)$, a contradiction.
  Therefore $s\in \mathcal{M}(\mathcal{R})$. Since $s\in
  [L]^{<\infty}$, we have that $s\in\mathcal{M}(\mathcal{R})
  \upharpoonright L$. Therefore
  $\mathcal{M}(\mathcal{R}\upharpoonright
  L)= \mathcal{M}(\mathcal{R})\upharpoonright
  L$.

Since $\mathcal{M}\upharpoonright
L=\mathcal{M}(\mathcal{R})\upharpoonright L\subseteq
\mathcal{R}\upharpoonright L$ and $\mathcal{R}\upharpoonright L$ is
hereditary, we have that $\widehat{\mathcal{M}\upharpoonright
L}\subseteq \mathcal{R}\upharpoonright L$. Conversely, let $s\in
\mathcal{R}\upharpoonright L$. Since $\mathcal{R}\upharpoonright L$
is compact there is $t\in\mathcal{M}(\mathcal{R}\upharpoonright
L)=\mathcal{M}(\mathcal{R})\upharpoonright L$ with $s\sqsubseteq t$.
Hence $s\in\widehat{\mathcal{M}\upharpoonright L}$ and
$\widehat{\mathcal{M}\upharpoonright L}=\mathcal{R}\upharpoonright
 L$.

 Finally, $o(\mathcal{M}\upharpoonright L)=o(\widehat{\mathcal{M}\upharpoonright
L})=o(\mathcal{R}\upharpoonright L)=o(\mathcal{R})$, where the last
equality follows by Proposition \ref{orderreg}.
\end{proof}
 \begin{cor}\label{fact every regular is large}
 Let $\ff$ be a regular thin family and $L\in[\nn]^\infty$. Then $\ff\upharpoonright L=\mathcal{M}(\widehat{\ff}\upharpoonright
 L)$, $\widehat{\ff\upharpoonright L}=\widehat{\ff}\upharpoonright
 L$ and $o(\ff\upharpoonright L)=o(\ff)$.
\end{cor}
\begin{proof} Since $\ff$ is thin we have that
$\ff=\mathcal{M}(\widehat{\ff})$.  Setting $\mathcal{R}=
\widehat{\ff}$ in  Lemma \ref{wx} the result follows.
\end{proof}
\begin{cor}\label{Galvin Pricley for regular thin}
Let $\ff$ be a regular thin family. Then for every
$M\in[\nn]^\infty$ there exists $L\in[M]^\infty$ such that
$\ff\upharpoonright L$ is very large in $L$.
\end{cor}

\begin{proof}
If $\ff$ is regular thin then since $\widehat{\ff}$ is spreading,
$\widehat{\ff}\upharpoonright N$ is nonempty for every $N\in
[\nn]^\infty$. Since $\widehat{\ff}\upharpoonright
N=\widehat{\ff\upharpoonright N}$, we get that $\ff\upharpoonright
N$ is nonempty too, i.e. $\ff$ is large in $\nn$. Therefore, by
Theorem \ref{Galvin prikry}, for every $M\in[\nn]^\infty$ there
exists $L\in[M]^\infty$ such that $\ff\upharpoonright L$ is very
large in $L$.
\end{proof}

 \begin{defn}For two families $\ff,\g$ of finite subsets of $\nn$, we write
 $\ff\sqsubseteq\g$  (resp. $\ff\sqsubset\g)$ if every element in $\ff$ has
an extension  (resp. proper extension) in $\g$ and
 every element in $\g$ has an  initial (resp. proper initial) segment in $\ff$.\end{defn}
  The following proposition is a
 consequence of a more general result from  \cite{GI}.

  \begin{prop}\label{corollary by Gasparis}
    Let $\ff,\g\subseteq[\nn]^{<\infty}$ be regular thin families with $o(\ff)<o(\g)$.
    Then for every $M\in[\nn]^\infty$ there exists
    $L\in[M]^\infty$ such that $\ff\upharpoonright L\sqsubset\g\upharpoonright L$.
  \end{prop}
  \begin{proof}
    By Corollary \ref{Galvin Pricley for regular thin} we have that there exists $L_1\in[M]^\infty$
    such that both $\ff,\g$ are very large in $L_1$. So for every $L\in[L_1]^\infty$ and
    every $t\in\g\upharpoonright L$ there exists $s\in\ff\upharpoonright L$ such that $s,t$ are comparable.

    Let $\g_1$ be the set of all elements of $\g$ which have a proper initial segment in $\ff$ and $\g_2=\g\setminus\g_1$.
    By Proposition \ref{th2} there exist $i_0\in\{1,2\}$ and $L\in[L_1]^\infty$
    such that $\g\upharpoonright L\subseteq \g_{i_0}$. It suffices to show
    that $i_0=1$. Indeed, if $i_0=2$ then for every $t\in G\upharpoonright L$ there is $s\in\ff$ such that $t\sqsubseteq s$.
    This in conjunction with  Corollary \ref{fact every regular is large} yields that $o(\g)=o(\g\upharpoonright L)\leq o(\ff)$ which is a contradiction.
  \end{proof}
  A similar but weaker result holds when  $o(\ff)=o(\g)$.
  \begin{prop}\label{px21}
 Let $\ff,\g\subseteq[\nn]^{<\infty}$ be regular thin families with $o(\ff)=o(\g)$.
 Then there exists $L_0\in[\nn]^\infty$ such that for every $M\in[\nn]^\infty$ there exists
 $L\in[L_0(M)]^\infty$ such that $L_0(\ff)\upharpoonright L\sqsubseteq\g\upharpoonright L$.
\end{prop}
\begin{proof}
By Proposition \ref{embending a family F in G by a shifting where
o(F)<o(G)} there exists  $L_0\in[\nn]^\infty$ such that
$L_0(\widehat{\ff})\subseteq\widehat{\g}$.
 Let $M\in [\nn]^\infty$. Notice that $L_0(\ff)$ and $\g$ are large in $L_0(M)$. Hence  by
Theorem \ref{Galvin prikry} there exists $L\in [L_0(M)]^\infty$ such
that $L_0(\ff)$ and $\g$ are  very large in $N$. Since
$L_0(\widehat{\ff})\subseteq\widehat{\g}$, we conclude that
$L_0(\ff)\upharpoonright L\sqsubseteq \g\upharpoonright L$.
\end{proof}
Technically the above  two propositions are incorporated in one as
 follows.
\begin{cor}\label{cx21}
 Let $\ff,\g\subseteq[\nn]^{<\infty}$ be regular thin families with $o(\ff)\leq o(\g)$.
 Then there exists $L_0\in[\nn]^\infty$ such that for every $M\in[\nn]^\infty$ there exists
 $L\in[L_0(M)]^\infty$ such that $L_0(\ff)\upharpoonright L\sqsubseteq\g\upharpoonright L$.
\end{cor}
\begin{proof}
If $o(\ff)< o(\g)$, we set $L_0=\nn$. Then $L_0(\ff)=\ff$ and
$L_0(M)=M$ and the conclusion follows by Proposition \ref{corollary
by Gasparis}. If $o(\ff)=0(\g)$  the result is immediate by
Proposition \ref{px21}.
\end{proof}
\section{Plegma families}\label{section admissibility} In this
section we introduce  the notion of   plegma families initially
appeared in \cite{AKT} for $k$-subsets of $\nn$. Here we do not
assume that the members of a plegma family are necessarily of the
same cardinality.
\subsection{Definition and basic properties}
We begin by stating  the definition of a plegma family.
  \begin{defn}\label{defn plegma}
    Let $l\in\nn$ and $s_1,...,s_l$ be nonempty finite subsets of $\nn$.
    The $l-$tuple $(s_j)_{j=1}^l$ will be called a \textit{plegma}
  family   if the following are satisfied.
    \begin{enumerate}
      \item[(i)] For every $i,j\in \{1,...,l\}$ and $k\in\nn$ with $i<j$ and $k\leq\min(|s_i|,|s_j|)$, we have that $s_i(k)<s_j(k)$.
      \item[(ii)] For every $i,j\in \{1,...,l\}$ and $k\in\nn$ with $k\leq \min (|s_i|,|s_j| -1)$, we have that $s_i(k)<s_j(k+1)$.
    \end{enumerate}
  \end{defn}
  For instance a pair $(\{n_1\},\{n_2\})$ of singletons is plegma iff
  $n_1<n_2$ and
   a pair of doubletons $(\{n_1,m_1\},\{n_2,m_2\})$ is plegma iff $n_1<n_2<m_1<m_2$.
More generally for two non empty $s,t\in [\nn]^{<\infty}$ with
$|s|\leq |t|$ the pair   $(s,t)$ is a plegma pair iff
$s(1)<t(1)<s(2)<t(2)<...<s(|s|)<t(|s|)$.  Of course  the situation
is more involved when the size of a plegma family is large.

Below we gather together some stability properties of plegma
families. We omit the proof as it is a direct application of the
definition.
   \begin{prop}\label{newq1}
Let  $(s_j)_{j=1}^l$ be a  family of finite subsets of $\nn$. Then
the following are satisfied.
\begin{enumerate}
\item[(i)] If  $(s_j)_{j=1}^l$ is a plegma family then
$(s_{j_m})_{m=1}^k$ is also a plegma family, for every $1\leq
k\leq l$ and $1\leq j_1<\ldots<j_k\leq l$.\item[(ii)] The family
$(s_j)_{j=1}^l$ is a plegma family iff $(s_{j_1},s_{j_2})$ is a
plegma pair, for every $1\leq j_1<j_2\leq l$. \item[(iii)] If
$(s_j)_{j=1}^l$ is a plegma family then $(t_j)_{j=1}^l$ is also
a plegma family, whenever $\varnothing\neq t_j\sqsubseteq s_j$
for $1\leq j\leq l$.
\item[(iv)] If $(s_j)_{j=1}^l$ is a plegma family then
$(L(s_j))_{j=1}^l$ is also a plegma family, for every
$L\in[\nn]^\infty$.
\end{enumerate}
\end{prop}
For every  family  $\ff\subseteq [\nn]^{<\infty}$ and $l\in\nn$ we
denote by $\text{{Plm}}_l(\ff)$ the set of all $(s_j)_{j=1}^l$
such that $s_1,...,s_l\in\ff$ and $(s_j)_{j=1}^l$ is a plegma
family. We also set $\text{{Plm}}(\ff)=\bigcup_{l=1}^\infty
\text{{Plm}}_l(\ff)$. Our mai aim is to  show that for every
$l\in\nn$, $\text{{Plm}}_l(\ff)$ is a Ramsey family. To this end
we need some preparatory lemmas.

\begin{lem} \label{lemma increasing length of plegma}
   Let $\ff$ be a regular thin family and $l\in\nn$. Then for every
 $(s_j)_{j=1}^l\in \text{{Plm}}_l(\ff)$ we have that
$|s_1|\leq\ldots\leq |s_l|$.
 \end{lem}
 \begin{proof}
   By (ii) of Proposition \ref{newq1} it suffices to show the conclusion  for $l=2$.
   Assume on the contrary that there exists a plegma pair $(s_1,s_2)$
   in $\ff$ with $|s_1|>|s_2|$. We pick $s\in[\nn]^{<\infty}$ such
   that $|s|=|s_1|$, $s_2\sqsubset s$ and $s(|s_2|+1)>\max s_1$.
   By the definition of the plegma family, we have that
   for every $1\leq k\leq|s_2|$, $s_1(k)<s_2(k)=s(k)$. Hence, for
   every $1\leq k\leq |s_1|$, we have that $s_1(k)\leq s(k)$. By
   the spreading property of $\widehat{\ff}$ we get that
   $s\in\widehat{\ff}$. But since $s_2$ is a proper initial segment of
   $s$ we get that $s_2\not\in\ff$, which is a contradiction.
 \end{proof}

  \begin{lem}\label{abmissibility1-1union and union thin proposition}    Let $\ff$ be a thin family of finite subsets of $\nn$
  and
$l\in\nn$.  Let   $(s_j)_{j=1}^l,(t_j)_{j=1}^l$ in
    $\text{Plm}_l(\ff)$ with $|s_1|\leq\ldots\leq |s_l|$,
    $|t_1|\leq\ldots\leq|t_l|$ and $\cup_{j=1}^l s_j\sqsubseteq\cup_{j=1}^l
    t_j$. Then
     $(s_j)_{j=1}^l=(t_j)_{j=1}^l$ and consequently $\cup_{j=1}^l s_j=\cup_{j=1}^l t_j$.
  \end{lem}

\begin{proof}
  Suppose that for some $1\leq m\leq l$ we have that $(s_i)_{i<m}=(t_i)_{i<m}$. We will show that $s_m=t_m$.
  Let $s=\cup_{j=m}^l s_j$ and $t=\cup_{j=m}^l t_j$. Then by our assumptions  $s\sqsubseteq t$.
  Moreover since $|s_m|\leq\ldots\leq|s_l|$ and $|t_m|\leq\ldots\leq|t_l|$, we easily conclude that
  $s_m(j)=s\big((j-1)(l-m+1)+1\big)$,  for all $1\leq j\leq|s_m|$ and similarly $t_m(j)=t\big((j-1)(l-m+1)+1\big)$,
  for all $1\leq j\leq|t_m|$. Hence, as $s\sqsubseteq t$, we get that  for all $1\leq j\leq \min\{|t_m|, |s_m|\},$ $s_m(j)=t_m(j)$.
  Therefore $s_m$ and $t_m$ are $\sqsubseteq$-comparable. Since  $\mathcal{F}$  is thin we have $s_m=t_m$.
  By induction on $1\leq m\leq l$, we obtain that  $s_j=t_j$ for every $1\leq j\leq l$.
\end{proof}

\begin{thm} \label{ramseyforplegma}
Let $M$ be an infinite subset of $\nn$, $l\in\nn$  and $\ff$ be a
regular thin family. Then for every finite partition
$\text{{Plm}}_l(\ff\upharpoonright M)=\cup_{i=1}^p \mathcal{P}_i$,
there exist $L\in[M]^\infty$ and $1\leq i_0\leq p$ such that
$\text{{Plm}}_l(\ff\upharpoonright L)\subseteq \mathcal{P}_{i_0}$.
\end{thm}
\begin{proof}
  Let
$\mathcal{U}=\{\cup_{j=1}^l
   s_j:(s_j)_{j=1}^l\in\text{{Plm}}_l(\ff\upharpoonright M)\}$.
By Lemmas \ref{lemma increasing length of plegma} and
\ref{abmissibility1-1union and union thin proposition} we get that
$\mathcal{U}$ is thin and the map $\Phi:
\text{{Plm}}_l(\ff\upharpoonright M)\to\mathcal{U} $ sending each
plegma $l$-tuple $(s_j)_{j=1}^l$ with $s_i\in\ff\upharpoonright M$
 for $1\leq i\leq l$, to its union $\cup_{j=1}^l s_j$,  is
 an onto  bijection.
We set $\mathcal{U}_i=\Phi(\mathcal{P}_i)$, for $1\leq j\leq p$.
Then $\mathcal{U}=\cup_{i=1}^p \mathcal{U}_i$ and  by Proposition
\ref{th2} there exist $j_0$ and $L\in[M]^\infty$ such that
$\mathcal{U}\upharpoonright L\subseteq _{i_0}\mathcal{U}_{i_0}$ or
equivalently $\text{{Plm}}_l(\ff\upharpoonright L)\subseteq
\mathcal{P}_{j_0}$.
\end{proof}
\subsection{Plegma paths}
In this subsection we introduce the definition of the plegma paths
in finite subsets of $\nn$ and we present some of  their
properties. Such paths will be used in the next subsection for the
study of plegma preserving maps.

\begin{defn}
  Let $k\in\nn$ and $s_0,...,s_k$ be nonempty finite subsets of $\nn$.
  We will say that $(s_j)_{j=0}^k$ is a \textit{plegma path} \textit{of length} $k$
  \textit{from} $s_0$ \textit{to} $s_k$, if for every $0\leq j\leq k-1$, the pair $(s_j,s_{j+1})$ is
  plegma. Similarly, a sequence $(s_j)_{j}$
  of nonempty finite subsets of $\nn$
  will be called  an infinite plegma path if for every
  $j\in\nn$ the pair $(s_j,s_{j+1})$ is plegma.
\end{defn}
The next simple lemma is very useful for the following.
\begin{lem}\label{lemma conserning the length of the plegma path}
 Let $(s_0,\ldots,s_{k-1},s)$ be a plegma path
  of length $k$ from $s_0$ to $s$ such that $s_0<s$. Then \[k\geq \min\{|s_i|:0\leq i\leq
  k-1\}.\]
\end{lem}
\begin{proof}
  Suppose that $k< \min\{|s_i|:0\leq i\leq k-1\}$. Then $s(1)<s_{k-1}(2)<s_{k-2}(3)<\ldots<s_1(k)<s_0(k+1)$, which contradicts that $s_0<s$.
\end{proof}
For a family $\ff\subseteq
  [\nn]^{<\infty}$ a \textit{plegma path in $\ff$} is a (finite or infinite) plegma
  path which consists of elements of $\ff$. It is easy to verify the existence of infinite plegma paths in $\ff$ whenever  $\ff$ is very large
   in an infinite subset $L$ of $\nn$.
   In particular, let $s\in\ff\upharpoonright L$ satisfying
   the next property: for every $j=1,...,|s|-1$ there exists $l\in L$ such that $s(j)<l<s(j+1)$.
   Then it is straightforward that there exists $s'\in\ff\upharpoonright L$ such that the pair $(s,s')$ is
    plegma and moreover $s'$ shares the same property with $s$. Based on this one can built an infinite
     plegma path in $\ff$  of elements having the above property.

     The above  remarks motivate the following definition.
For every $\ff\subseteq [\nn]^{<\infty}$  and $L\in[\nn]^\infty$, we
set\begin{equation}\ff\upharpoonright\upharpoonright
L=\Big{\{}s\in\ff\upharpoonright L:\forall 1\leq j\leq |s|-1\
\exists \ l\in L\text{ with }s(j)<l<s(j+1)\Big{\}}.\end{equation}

The proof of the next lemma follows the same lines with the one of
 Lemma \ref{wx}.
\begin{lem}\label{wq}
 Let $\ff$ be a regular thin family and $L\in[\nn]^\infty$. Then
 $\widehat{\ff\upharpoonright\upharpoonright L}=\widehat{F}\upharpoonright\upharpoonright
 L$, i.e. $s\in\ff\upharpoonright\upharpoonright L$ iff $s$ is   $\sqsubseteq$-maximal
 in $\widehat{F}\upharpoonright\upharpoonright
 L$.
 \end{lem}
 We are now ready to present the main result of this subsection.  In terms of graph theory it states
that in the (directed) graph with vertices the elements of
$\ff\upharpoonright\upharpoonright L$ and edges the plegma pairs
$(s,t)$ in $\ff\upharpoonright\upharpoonright L$, the distance
between two vertices $s_0$ and $s$ with $s_0<s$ is equal to the
cardinality of $s_0$.

\begin{thm}\label{accessing everything with plegma path of length |s_0|}
  Let $\ff$ be a regular thin family and $L\in[\nn]^\infty$. Assume
  that $\ff$ is very large in $L$.
  Then for every $s_0,s\in\ff\upharpoonright\upharpoonright L$ with $s_0<s$
   there exists a plegma path $(s_0,\ldots,s_{k-1},s)$ in $\ff\upharpoonright\upharpoonright L$
  of length $k=|s_0|$ from $s_0$ to $s$. Moreover  $k=|s_0|$ is
  the minimal length of a plegma path in $\ff\upharpoonright\upharpoonright L$ from $s_0$ to $s$.
\end{thm}
\begin{proof}
By Lemmas \ref{lemma conserning the length of the plegma path} and
\ref{lemma increasing length of plegma} every plegma path in $\ff$
from $s_0$ to $s$ is of length at least $|s_0|$. Therefore for
$s_0<s$ a  plegma path of the form $(s_0,\ldots,s_{k-1},s)$ with
$s_0,...,s_{k-1},s\in\ff$ and $k=|s_0|$ certainly is of minimal
length.

We will actually prove a slightly more general result.  Namely we
will show that for every $t$ in
$\widehat{\ff}\upharpoonright\upharpoonright L$ and
$s\in\ff\upharpoonright\upharpoonright L$ with $t<s$ there exists a
plegma path of length $|t|$ from $t$ to $s$ such that all its
elements except perhaps   $t$ belong to
$\ff\upharpoonright\upharpoonright L$.

For the proof we will use induction on the length of $t$. The case
$|t|=1$ is trivial, since for every $s\in [\nn]^{<\infty}$ with
$t<s$ the pair $(t,s)$ is already a plegma path of length 1 from $t$
to $s$. Suppose that for some $k\in\nn$ the above holds for all $t$
in $\widehat{\ff}\upharpoonright\upharpoonright L$ with $|t|=k$.

Let $t\in \widehat{\ff}\upharpoonright\upharpoonright L$ with
$|t|=k+1$
  and $s\in\ff\upharpoonright\upharpoonright L$ with $t<s$. Then there exist
  $n_1<n_2<\ldots<n_{k+1}$ in $\nn$ such that $n_{j}-n_{j-1}>1$, for  $2\leq j\leq k$ and $t=\{L(n_j):1\leq j\leq
  k+1\}$. We set $t_0=\{L(n_j-1):2\leq j\leq k+1\}$. Since $n_j-1>n_{j-1}$ we have that $t_0$ is
  of equal cardinality and pointwise strictly greater than $t\setminus \{\text{max}\ t\}$. Hence, since $\widehat{\ff}$ is spreading,
  we have  $t_0\in\widehat{\ff}\upharpoonright\upharpoonright L$ and
  moreover $t_0$ cannot be a $\sqsubseteq$-maximal element
  of $\widehat{\ff}\upharpoonright\upharpoonright
  L$.  By Lemma \ref{wq} we have that $\ff\upharpoonright\upharpoonright
  L$ is the set of all $\sqsubseteq$-maximal elements of $\widehat{\ff}\upharpoonright\upharpoonright
  L$. Therefore, we conclude that $t_0\in \widehat{\ff}\setminus
  \ff$.
Thus, since $|t_0|=k$, by the inductive
  hypothesis, there exists a plegma path $(t_0,s_1,\ldots,s_{k-1},s)$
  of length $k=|t_0|$ from $t_0$ to $s$ with
  all $s_1,\ldots,s_{k-1},s$  in $\ff\upharpoonright\upharpoonright L$.

  Let $l=|s_1|$. Since $(t_0,s_1)$ is a plegma pair with $t_0\in\widehat{F}\setminus \ff$ and
  $s_1\in\mathcal{F}$, arguing as in Lemma \ref{lemma increasing length of plegma}, we see that $l\geq k+1$.
  Moreover since $s_1\in \ff\upharpoonright\upharpoonright L$,  there exist   $m_1<\ldots<m_l$ in $\nn$
  such that $m_j-m_{j-1}>1$ and $s_1=\{L(m_j):1\leq j\leq l\}$. Notice that  $n_2\leq m_1< n_3\leq m_2< ...< n_k\leq m_{k-1}<m_{k+1}-1$.

  We set $w=t_0\cup\{L(m_j-1):k+1\leq j\leq l\}$ and let
$L'\in[L]^\infty$ such that $w$ is an initial segment of $L'$.
Notice that $|w|=l$. Since $\ff$ is very large in $L$ there exists
$s_0\in \ff$ with $s_0$ initial segment of $L'$. Using again that
$\widehat{\ff}$ is spreading it is  shown that $|t_0|<|s_0|\leq l$
and therefore  $t_0\sqsubset s_0\sqsubseteq w$.

It is easy to check that  $(t,s_0)$ and $(s_0,s_1)$ are plegma
pairs. Hence the sequence $(t,s_0,\ldots,s_{k-1},s)$ is a plegma
path of length $k+1$ from $t$ to $s$ with
$s_0,\ldots,s_{k-1},s\in\ff\upharpoonright\upharpoonright L$. The
proof of the inductive step as well as of the theorem is complete.
\end{proof}

We close this section by presenting an  application of the above
theorem.  We start with the following definition.

   Let $X$ be a set, $M\in[\nn]^\infty$,  $\ff\subseteq
[\nn]^{<\infty}$ and $\varphi:\ff\to X$.
    We will say that $\varphi$ is \textit{hereditarily nonconstant} in $M$ if
    for every $L\in[M]^\infty$ the restriction of $\varphi$ on
    $\ff\upharpoonright L$ is nonconstant. In particular if $M=\nn$ then we will simply  say that $\varphi$ is
    hereditarily nonconstant.

  \begin{lem}\label{lemma making a hereditary nonconstant function, nonconstant on plegma pairs}
    Let $\ff$ be a regular thin family, $X$ be a set  and
    $\varphi:\ff\to X$ be hereditarily nonconstant.
    Then for every $N\in[\nn]^\infty$ there exists $L\in [N]^\infty$ such that for every plegma pair
    $(s_1,s_2)$ in $\ff\upharpoonright L$,  $\varphi(s_1)\neq\varphi(s_2)$.
  \end{lem}
  \begin{proof}
    By Theorem \ref{ramseyforplegma} there exists an $L\in [N]^\infty$
    such that either $\varphi(s_1)\neq\varphi(s_2)$, for all plegma pairs
    $(s_1,s_2)$ in $\ff\upharpoonright L$, or $\varphi(s_1)=\varphi(s_2)$ for all plegma pairs
    $(s_1,s_2)$ in $\ff\upharpoonright L$. The second alternative is excluded.
    Indeed, suppose that  $\varphi(s_1)=\varphi(s_2)$, for every plegma pair
    $(s_1,s_2)$ in $\ff\upharpoonright L$.
    By Corollary \ref{Galvin Pricley for regular thin} we may also assume that $\ff\upharpoonright L$ is very large in $L$.
    Let  $s_0$  be the unique   initial segment of  $L_0=\big{\{} L(2\rho):\rho\in\nn\big{\}}$ in
     $\ff\upharpoonright L$ and  let $k=|s_0|$.
     We set
    $L_0'=\big{\{} L(2\rho):\rho\in\nn\text{ and } \rho>k \big{\}}$. By
    Theorem
    \ref{accessing everything with plegma path of length |s_0|} for every
    $s\in\ff\upharpoonright L_0'$ there exist a plegma path $(s_0,s_1,\ldots,s_{k-1},s)$
    of length $k$ in $\ff\upharpoonright L$. Therefore for every $s\in\ff\upharpoonright L_0'$
    we have that $\varphi(s)=\varphi(s_{k-1})=\ldots=\varphi(s_1)=\varphi(s_0)$, which
    contradicts that $\varphi$ is hereditarily nonconstant.
  \end{proof}
\begin{prop}
  Let $\ff$ be a regular thin family, $M\in[\nn]^\infty$ and $\varphi:\ff\to \nn$
 be hereditarily non constant in $M$. Let also $g:\nn\to\nn$.
 Then there exists $N\in[M]^\infty$ such that for every plegma
pair $(s_1,s_2)$ in $\ff\upharpoonright N$,
$\varphi(s_2)-\varphi(s_1)>g(n)$, where $\min s_2=N(n)$.
\end{prop}
\begin{proof}
  By Theorem \ref{ramseyforplegma} there exists $L\in[M]^\infty$
  such that one of following holds.
  \begin{enumerate}
    \item[(i)] For every plegma pair
    $(s_1,s_2)$ in $\ff\upharpoonright L$, we have
    $\varphi(s_1)=\varphi(s_2)$.
    \item[(ii)] For every plegma pair
    $(s_1,s_2)$ in $\ff\upharpoonright L$, we have
    $\varphi(s_1)>\varphi(s_2)$.
    \item[(iii)] For every plegma pair
    $(s_1,s_2)$ in $\ff\upharpoonright L$, we have
    $\varphi(s_1)<\varphi(s_2)$.
  \end{enumerate}
  Since  $\varphi$
 is hereditarily non constant in $M$, by Lemma \ref{lemma making a hereditary nonconstant function, nonconstant on plegma pairs}
 case (i) is
  excluded.   Similarly case (ii) cannot occur since otherwise,
  $(\varphi(s_n))_{n}$ would form a strictly decreasing sequence
  in $\nn$ whenever $(s_n)_n$ is an infinite  plegma path in $\ff\upharpoonright L$. Therefore, case (iii)
  holds. We choose $N\in[L]^\infty$ such that for every $n\geq2$,
  we have that \[\Big|\Big\{l\in L:N(n-1)<l<N(n)\Big\}\Big|\geq\max_{j\leq n}g(j).\]
  Let $(s_1,s_2)$ be a plegma pair in $\ff\upharpoonright N$ and let $n\in\nn$ such that  $\min s_2=N(n)$.
   Notice for every $1\leq k\leq |s_1|$, we
  have  \[\Big|\Big\{l\in L:s_1(k)<l<s_2(k)\Big\}\Big|\geq
  g(n).\] Similarly for every $|s_1|<k\leq |s_2|$, we have
  \[\Big|\Big\{l\in L:s_2(k-1)<l<s_2(k)\Big\}\Big|\geq
  g(n).\]
  The above  yield that there exist
  $t_1,\ldots,t_{g(n)}\in\ff\upharpoonright L$ such that the
  $(g(n)+2)$-tuple $(s_1,t_1,\ldots,t_{g(n)},s_2)$ is plegma. Hence $\varphi(s_2)-\varphi(s_1)>g(n)$.
\end{proof}
  \begin{cor}\label{Proposition combinatorial for SSD}
    Let $\ff$ be a regular thin family, $M\in[\nn]^\infty$ and $\varphi:\ff\to \nn$
 be hereditarily non constant in $M$.
Then  there exists $N\in[M]^\infty$ such that for every plegma pair
$(s_1,s_2)$ in $\ff\upharpoonright N$ we have
$\varphi(s_2)-\varphi(s_1)>1$.
\end{cor}
\subsection{Plegma preserving maps}
 Let $\ff\subseteq[\nn]^{<\infty}$ and $\varphi:\ff\to
  [\nn]^{<\infty}$. We will say that the map $\varphi$ is \textit{plegma
  preserving} if $(\varphi(s_1), \varphi(s_2))$ is a plegma pair whenever $(s_1,s_2)$ is a plegma pair in
  $\ff$.

\begin{lem}\label{ququ}  Let $\ff\subseteq[\nn]^{<\infty}$  and
  $\varphi:\ff\to[\nn]^{<\infty}$.
 If $\varphi$ is plegma preserving then for every  $l\in\nn$ and $(s_j)_{j=1}^l\in \emph{\text{Plm}}(\ff)$
we have that $(\varphi(s_j))_{j=1}^l$ is a plegma $l$-tuple.
\end{lem}
\begin{proof} Let $l\in\nn$ and $(s_j)_{j=1}^l$ be a plegma $l$-tuple in
$\ff$. Then for every $1\leq j_1<j_2\leq l$ we have that
$(s_{j_1},s_{j_2})$ is plegma and thus
$(\varphi(s_{j_1}),\varphi(s_{j_2}))$ is plegma. Hence, by (ii) of
Proposition \ref{newq1}, $(\varphi(s_j))_{j=1}^l$ is a plegma
$l$-tuple.
\end{proof}
\begin{prop}\label{bfk}
  Let $\ff$ be a regular thin family and
  $\varphi:\ff\to[\nn]^\infty$. Then for every $M\in[\nn]^\infty$
  there is $L\in[M]^\infty$ such that exactly one of the following holds.  \begin{enumerate}
    \item[(i)] The restriction of $\varphi$ on $\ff\upharpoonright L$
    is  plegma preserving.
    \item[(ii)] For every  $(s_1,s_2)\in\text{{Plm}}_2(\ff\upharpoonright L)$ neither
    $(\varphi(s_1),\varphi(s_2))$ nor $(\varphi(s_2),\varphi(s_1))$ is a plegma pair.
  \end{enumerate}
\end{prop}
\begin{proof} Assume that there is  $M\in[\nn]^\infty$ such that for every   $L\in [M]^\infty$
 neither (i) nor (ii) holds true. Then by Theorem
\ref{ramseyforplegma} there exists $L\in
  [M]^\infty$ such that for every  $(s_1,s_2)\in\text{{Plm}}_2(\ff\upharpoonright
  N)$ we have that  $(\varphi(s_2),\varphi(s_1))$ is  plegma.
But this is impossible. Indeed, otherwise  for an infinite plegma
path $(s_n)_{n}$ in $\ff\upharpoonright N$ the sequence $(\min
s_n)_{n}$ would form a strictly decreasing infinite sequence in
$\nn$.
\end{proof}
For a family $\ff\subseteq [\nn]^{<\infty}$ and a plegma preserving
map $\varphi:\ff\to [\nn]^{<\infty}$ we will say that $\varphi$ is
\textit{normal} provided that  $|\varphi(s_1)|\leq|\varphi(s_2)|$
for every plegma pair $(s_1,s_2)$ in $\ff$  and
$|\varphi(s)|\leq|s|$ for every $s\in\ff$.

\begin{thm}\label{Properties of plegma preserving map}
  Let $\ff$ be a regular thin family, $M\in[\nn]^\infty$ and $\varphi:\ff\upharpoonright M\to[\nn]^{<\infty}$
  be a plegma preserving map.
  Then there exists $L\in[M]^\infty$ such that the restriction
  of $\varphi$ on $\ff\upharpoonright L$ is  a normal plegma preserving map.
\end{thm}
\begin{proof}
 By  Theorem \ref{ramseyforplegma} there exists $N\in [M]^\infty$ such that either
(a)  $|\varphi(s_1)|\leq|\varphi(s_2)|$, for every plegma pair
$(s_1,s_2)$ in $\ff\upharpoonright N$, or (b)
$|\varphi(s_1)|>|\varphi(s_2)|$, for every plegma pair $(s_1,s_2)$
in $\ff\upharpoonright N$.
  Alternative (b) cannot occur since otherwise for an infinite plegma path  $(s_n)_{n}$
   in $\ff\upharpoonright N$ the sequence $(|\varphi(s_n)|)_{n}$ would form  a strictly decreasing sequence in $\nn$.
   By  Theorem \ref{th2} there exists
    $L\in[N]^\infty$ such that either
(c) $|\varphi(s)|\leq|s|$, for every $s\in\ff\upharpoonright
    L$,  or (d) $|\varphi(s)|>|s|$, for every
  $s\in\ff\upharpoonright L$. We claim that
    (d) cannot hold true.  Indeed, since $\varphi$ on
$\ff\upharpoonright L$ is plegma preserving, using a plegma path of
enough large length, we may choose
   $s_0, s$ in $\ff\upharpoonright\upharpoonright L$ such that
    $\min s_0 <\min s$ and $\min \varphi(s_0)<\min \varphi(s)$.
   Let $k_0=|s_0|$. Then by
    Proposition \ref{accessing everything with plegma path of length
    |s_0|} there exists a plegma path $(s_i)_{i=0}^{k_0}$ in
    $\ff\upharpoonright\upharpoonright L$  from $s_0$ to
    $s=s_{k_0}$ of length $k_0$. By Lemma \ref{ququ}  $(\varphi(s_i))_{i=0}^{k_0}$ is also a plegma path of
     length $k_0$ from $\varphi(s_0)$ to $\varphi(s_{k_0})$ and
    by Lemma \ref{lemma conserning the length of the plegma path} we
have that
     \begin{equation}\label{b1}\min\{|\varphi(s_i)|:0\leq i\leq k_0-1\}\leq
     k_0.\end{equation}
     Moreover by Lemma \ref{abmissibility1-1union and union thin
     proposition}, we have
     $|s_0|\leq |s_1|\leq ... \leq |s_{k_0}|$.
Hence if  (d) holds true then
      \begin{equation}\min\{|\varphi(s_i)|:0\leq i\leq k_0-1\}> \min\{|s_i|:0\leq i\leq k_0-1\}\geq
      k_0,\end{equation}
      which  contradicts (\ref{b1}).

      Therefore we conclude that   (a) and (c) hold true i.e. the restriction of $\varphi$
on $\ff\upharpoonright L$ is a normal plegma preserving map.
\end{proof}
\subsection{Plegma preserving maps between thin families} In this
subsection we are concerned with the  question of the existence of
a plegma preserving map $\varphi:\g\to\ff$, where $\g$ and $\ff$
are regular thin families. We shall show that  such maps exist
only when $o(\g)\geq o(\ff)$. We start with the positive result.
\begin{thm}\label{plegma preserving maps}
  Let $\ff,\g$ be regular thin families with $o(\ff)\leq o(\g)$.
  Then
  for every $M\in[\nn]^\infty$ there is $N\in[\nn]^\infty$ and
  a plegma preserving map
  $\varphi:\g\upharpoonright N\to\ff\upharpoonright M$. Moreover, for every $l\in\nn$ and $t\in\g\upharpoonright N$, if
$\min t\geq N(l)$ then $\min\varphi(t)\geq M(l)$.
\end{thm}

\begin{proof}
 Let  $M\in[\nn]^\infty$.   By Corollary \ref{cx21} there exists  $L_0\in[\nn]^\infty$
and  $N\in [L_0(M)]^\infty$ such that  $L_0(\ff)\upharpoonright
N\sqsubseteq \g\upharpoonright N$. Thus for every
$t\in\g\upharpoonright N$ there exists a unique  $s_t\in\ff$ such
that $L_0(s_t)\sqsubseteq t$. Moreover, $L_0(s_t)\sqsubseteq
t\subseteq N\subseteq L_0(M)$ and therefore $s_t\subseteq M$. We
define  $\varphi:\g\upharpoonright N\to\ff\upharpoonright M$, by
setting $\varphi(t)=s_t$.  To see that $\varphi$ is plegma
preserving, let $(t_1,t_2)$ be a plegma pair in $\g\upharpoonright
N$. Then $L_0(\varphi(t_i))\sqsubseteq t_i$, for $i\in\{1,2\}$ and
therefore by (iv) of Proposition \ref{newq1} $(L_0(\varphi(t_1),
(L_0(\varphi(t_2))$ and  $(\varphi(t_1), \varphi(t_2))$ are  also
plegma pairs. Hence $\varphi$ is plegma preserving.

Finally, let $l\in\nn$ and $t\in\g\upharpoonright N$ with $\min
t\geq N(l)$. Since $L_0(\varphi(t))\sqsubseteq t$, we have that
$\min L_0(\varphi(t))=\min t$ and therefore $L_0(\min
\varphi(t))=\min L_0(\varphi(t))=\min t\geq N(l)\geq
L_0(M)(l)=L_0(M(l))$. Hence $\min \varphi(t)\geq M(l)$.
\end{proof}
For the following  we shall need the next definition. Let
$\ff\subseteq[\nn]^{<\infty}$ and $L\in[\nn]^\infty$. We define
  \begin{equation} L^{-1}(\ff)=\Big{\{}
  t\in[\nn]^{<\infty}:L(t)\in\ff\Big{\}}.\end{equation}

 It is easy to see that for every family $\ff\subseteq [\nn]^{<\infty}$  and $L\in[\nn]^\infty$ the following are satisfied.
  \begin{enumerate}
    \item[(a)] If $\ff$ is very large in $L$ then the family $L^{-1}(\ff)$ is very large in $\nn$.
    \item[(b)] If $\ff$ is regular thin then so does the family $L^{-1}(\ff)$.
    \item[(c)] $o(L^{-1}(\ff))=o(\ff\upharpoonright L)$. In
    particular if $\ff$ is regular thin then
    $o(L^{-1}(\ff))=o(\ff)$.
  \end{enumerate}

\begin{lem}\label{firsttheoreminadm}
  Let $\ff$ be a regular thin family, $L\in[\nn]^\infty$  such that $\mathcal{F}$ is very large in
  $L$.
  Let $\varphi:\ff\upharpoonright L\to[\nn]^{<\infty}$ be a normal plegma preserving
  map. Let $\psi:L^{-1}(\ff)\to [\nn]^{<\infty}$ defined by $\psi(u)=\varphi(L(u))$ for every $u\in L^{-1}(\ff)$.
  Then $\psi$ is a normal plegma preserving map which in addition satisfies the following property. If
  $u\in L^{-1}(\ff)\upharpoonright\upharpoonright\nn$ and $w=\psi(u)$ then  $u(i)\leq w(i)$ for every $1\leq i\leq |w|$.
\end{lem}
\begin{proof} It is easy to check that $\psi$ is a normal plegma preserving
  map. Therefore we pass to the proof of the property of $\psi$. First, by induction on $k=u(1)$, we shall show that
$u(1)\leq\psi(u)(1)$,  for all $u\in
L^{-1}(\ff)\upharpoonright\upharpoonright\nn$. Indeed, if $u(1)=1$
then obviously $\psi(u)(1)\geq 1=u(1)$. Suppose that for some
$k\in\nn$ and every $u\in
L^{-1}(\ff)\upharpoonright\upharpoonright\nn$ with $u(1)=k$ we have
that $\psi_1(u)(1)\geq k$. Let $u\in
L^{-1}(\ff)\upharpoonright\upharpoonright\nn$ with $u(1)=k+1$. Since
$L^{-1}(\ff)$ is regular thin and very large in $\nn$, we easily see
that there exists a unique $u'\in L^{-1}(\ff)$ with $u'\sqsubseteq
\{u(\rho)-1: 1\leq \rho\leq |u|\}$. Notice that $(u',u)$ is a plegma
pair in $L^{-1}(\ff)$ and $u'(1)=k$. Since $\psi$ is a normal plegma
preserving map we have that $(\psi(u'),\psi(u))$ is also a plegma
pair. Hence $\psi(u)(1)>\psi(u')(1)\geq u'(1)=k$, that is
$\psi(u)(1)\geq k+1=u(1)$.

Suppose now that for some $i\in\nn$ and  every $u\in
L^{-1}(\ff)\upharpoonright\upharpoonright\nn$ with $i\leq
|\psi(u)|$, $u(i)\leq\psi(u)(i)$. Let $u\in
L^{-1}(\ff)\upharpoonright\upharpoonright\nn$ with $i+1\leq
|\psi(u)|$. Since $L^{-1}(\ff)$ is  very large in $\nn$, there
exists $u'\in L^{-1}(\ff)\upharpoonright\upharpoonright\nn$ such
that $\{u(\rho)-1: 2\leq \rho\leq |u|\}\sqsubseteq u'$. Observe that
$(u,u')$ is plegma pair in $L^{-1}(\ff)$, $|u|\leq |u'|$ and
$u(i+1)=u'(i)+1$. Since $\psi$ is normal plegma preserving, we have
that $(\psi(u),\psi(u'))$ is also a plegma pair and in addition
$i+1\leq |\psi(u)|\leq|\psi(u')|$. Hence,
$\psi(u)(i+1)>\psi(u')(i)\geq u'(i)=u(i+1)-1$, that is
$\psi(u)(i+1)\geq u(i+1)$.
 By induction on $i\in\nn$ the proof  is complete.
\end{proof}

\begin{thm}\label{non plegma preserving maps}
    Let $\ff,\g$ be regular thin families with  $o(\ff)<o(\g)$ and
    $M\in [\nn]^\infty$.
    Then there is no plegma preserving map from $\ff\upharpoonright M$ to $\g$. More precisely for every
     $M\in[\nn]^\infty$ and $\varphi:\ff\upharpoonright M\to\g$ and there exists $L\in[M]^\infty$ such that for every
    plegma pair $(s_1,s_2)$ in $\ff\upharpoonright L$ neither $(\phi(s_1),\phi(s_2))$ nor $(\phi(s_2),\phi(s_1))$ is a plegma pair.
  \end{thm}

  \begin{proof}
   Assume  that there exist $M\in[\nn]^\infty$ and  $\varphi:\ff\upharpoonright M\to\g$
    such that $\varphi$ is  plegma preserving.
    By Theorem \ref{Properties of plegma preserving map} there exists
  $L\in[M]^\infty$ such that the restriction
   of $\varphi$ on $\ff\upharpoonright L$ is a normal plegma preserving map.
  By Corollary \ref{Galvin Pricley for regular thin} we may also assume that
  $\ff$ is very large in $L$.
   Since $o(L^{-1}(\ff))=o(\ff)<o(\g)$ by
Proposition \ref{corollary by Gasparis} we have that there exists
$N\in[\nn]^\infty$ such that $L^{-1}(\ff)\upharpoonright N\sqsubset
\g\upharpoonright N$. We may  assume that $N(i+1)-N(i)>1$ and
therefore  $L^{-1}(\ff)\upharpoonright N\subseteq
L^{-1}(\ff)\upharpoonright\upharpoonright\nn$. Let
$\psi:L^{-1}(\ff)\to [\nn]^{<\infty}$
  defined by $\psi(u)=\varphi(L(u))$ for every $u\in L^{-1}(\ff)$.

  Pick $u_0\in
L^{-1}(\ff)\upharpoonright N$ and set $w_0=\psi(u_0)$. Since
$L^{-1}(\ff)\upharpoonright N\sqsubset \g\upharpoonright N$, we have
that $u_0\in\widehat{\g}\setminus \g$ and since  $\varphi$ takes
values in $\mathcal{G}$, we have that $w_0\in\g$. We are now ready
to derive a contradiction. Indeed, by Lemma \ref{firsttheoreminadm}
we have that $\psi$ is a normal plegma preserving map which implies
that $|w_0|\leq |u_0|$. Moreover, since $L^{-1}(\ff)\upharpoonright
N\subseteq L^{-1}(\ff)\upharpoonright\upharpoonright\nn$, we have
that $u_0\in L^{-1}(\ff)\upharpoonright\upharpoonright\nn$. Hence,
again by Lemma \ref{firsttheoreminadm}, we get that  $ u_0(i)\leq
w_0(i)$, for every $1\leq i\leq |w_0|$. Summarizing we have that
 $u_0\in \widehat{\g}\setminus \g$, $|w_0|\leq |u_0|$ and
$ u_0(i)\leq w_0(i)$, for every $1\leq i\leq |w_0|$. Since
$\widehat{\g}$ is spreding we conclude that $w_0\in
\widehat{\g}\setminus \g$, which is impossible. Therefore there is
no $M\in [\nn]^\infty$ and $\varphi:\ff\upharpoonright M\to \g$ such
that $\varphi$ is plegma preserving. By Proposition \ref{bfk}, we
get that for every $M\in [\nn]^\infty$ and
$\varphi:\ff\upharpoonright M\to \g$ there is $L\in[M]^\infty$ such
that for every plegma pair $(s_1,s_2)$ in $\ff\upharpoonright L$
neither $(\phi(s_1),\phi(s_2))$ nor $(\phi(s_2),\phi(s_1))$ is
plegma.
\end{proof}

\section{The hierarchy of spreading models}\label{Chapter 2}
In this section we  define  the class of the $\xi$-spreading models
of a Banach space $X$ for every countable ordinal $1\leq
\xi<\omega_1$. The definition is a transfinite extension of the
corresponding one  of  the finite order spreading models given in
\cite{AKT}.  The basic ingredients of this extension are the
concepts of the $\ff$-sequences in $X$, i.e. sequences of the form
$(x_s)_{s\in\mathcal{F}}$ with $x_s\in X$ for every
$s\in\mathcal{F}$ and the plegma families with members in $\ff$
where $\ff$ is a regular thin family.
\subsection{The $\ff$-spreading
models of a Banach space $X$} Let  $X$ be a Banach space and
$\ff\subseteq [\nn]^{<\infty}$ be a regular thin family. By the
term \emph{$\ff$-sequence} in $X$ we will mean  a map
$\varphi:\ff\to X$. An $\ff$-sequence in $X$ will be usually
denoted by $(x_s)_{s\in\ff}$, where $x_s=\varphi(s)$ for all
$s\in\ff$. Also, for every $M\in[\nn]^\infty$, the map $\varphi:
\ff\upharpoonright M\to X$ will be called an
\emph{$\ff$-subsequence} of $(x_s)_{s\in\ff}$ and will be denoted
by $(x_s)_{s\in\ff\upharpoonright M}$. An $\ff$-sequence
$(x_s)_{s\in \mathcal{F}}$ in $X$ will be called bounded (resp.
seminormalized) if there exists $C>0$ (resp. $0<c<C$) such that
$\|x_s\|\leq C$ (resp. $c\leq \|x_s\|\leq C$) for every $s\in\ff$.
\begin{lem}\label{ljy}
  Let $(x_s)_{s\in\ff}$ be a bounded
$\ff$-sequence  in  $X$. Let $k\in\nn$, $N\in[\nn]^\infty$ and
$\delta>0$. Then there exists $M\in[N]^\infty$ such that
  \begin{equation}\label{eqsz1}\Bigg{|}\Big{\|}\sum_{j=1}^ka_jx_{t_j}\Big{\|}-\Big{\|}\sum_{j=1}^ka_jx_{s_j}\Big{\|}\Bigg{|}\leq\delta\end{equation}
  for every
  $(t_j)_{j=1}^k, (s_j)_{j=1}^k \in\text{{Plm}}_k(\ff\upharpoonright M)$ and
  $a_1,...,a_k\in[-1,1]$.
\end{lem}
\begin{proof}
Let $(\textbf{a}_n)_{n=1}^{n_0}$ be a $\frac{\delta}{3l}-$net of
the unit ball of $(\rr^k,\|\cdot\|_\infty)$. Setting
$\textbf{a}_n=(a_1^n,...,a_k^n)$ for every $1\leq n\leq n_0$, we
inductively construct  $N=N_0\supseteq N_1\supseteq\ldots\supseteq
N_{n_0}$ satisfying
  \begin{equation}\label{wq1}\Bigg{|}\Big\|\sum_{j=1}^ka_j^nx_{t_j}\Big\|-\Big\|\sum_{j=1}^ka_j^nx_{s_j}\Big\|\Bigg{|}\leq\frac{\delta}{3},
  \end{equation}
for every $1\leq n\leq n_0$ and every
  $(s_j)_{j=1}^k,(t_j)_{j=1}^k\in\text{{Plm}}_k(\ff\upharpoonright
N_{n})$.

The inductive step of the construction has as follows. Suppose
that $N_0,\ldots,N_{n-1}$ have been constructed. Define
$g_n:\text{{Plm}}_k(\ff\upharpoonright N_{n-1})\to[0,lC]$ by
$g_n((s_j)_{j=1}^k)=\|\sum_{j=1}^k a_j^n x_{s_j}\|$. By dividing
the interval $[0,lC]$ into disjoint intervals of length
$\frac{\delta}{3}$ and applying Theorem \ref{ramseyforplegma},
there is  $N_n\subseteq N_{n-1}$ such that
$|g_n((t_j)_{j=1}^k)-g_n((s_j)_{j=1}^k)|\leq\frac{\delta}{3}$, for
every
$(t_j)_{j=1}^k,(s_j)_{j=1}^k\in\text{{Plm}}_k(\ff\upharpoonright
N_n)$.

We set $M= M(N_{n_0})$. By (\ref{wq1}) and taking into account that
$(\textbf{a}_n)_{n=1}^{n_0}$ is a $\frac{\delta}{3}-$net of the unit
ball of $(\rr^k, \|\cdot\|_\infty)$ it is easy to see that $L$ is as
desired.
\end{proof}

\begin{lem}\label{jh}
  Let $(x_s)_{s\in\ff}$ be a bounded
$\ff$-sequence  in  $X$. Let $l\in\nn$, $N\in[\nn]^\infty$ and
$\delta>0$. Then there exists $M\in[N]^\infty$ such that
  \begin{equation}\label{eqsz2}\Bigg{|}\Big{\|}\sum_{j=1}^ka_jx_{t_j}\Big{\|}-\Big{\|}\sum_{j=1}^ka_jx_{s_j}\Big{\|}\Bigg{|}\leq\delta\end{equation}
  for every $1\leq k\leq l$,
  $(t_j)_{j=1}^k, (s_j)_{j=1}^k \in\text{{Plm}}_k(\ff\upharpoonright M)$ and
  $a_1,...,a_k\in[-1,1]$.
\end{lem}
\begin{proof}
It follows easily by an iterated use of   Lemma \ref{ljy}.
\end{proof}
\begin{lem} \label{B-Sp} Let  $(x_s)_{s\in\ff}$ be a bounded $\ff$-sequence in $X$. Then for every
  sequence $(\delta_n)_n$ of positive real numbers and  $N\in[\nn]^\infty$ there exists
$M\in[N]^\infty$ satisfying
\begin{equation}\label{pot}\Bigg{|}\Big\|\sum_{j=1}^k
a_jx_{t_j}\Big\|-\Big\|\sum_{j=1}^k
a_jx_{s_j}\Big\|\Bigg{|}\leq\delta_l,\end{equation} for every
$1\leq k\leq l$, $a_1,...,a_k\in[-1,1]$ and $(t_j)_{j=1}^k,
(s_j)_{j=1}^k \in\text{{Plm}}_k(\ff\upharpoonright M)$ such that
$s_1(1),t_1(1)\geq M(l)$.
\end{lem}
\begin{proof} It is straightforward  by Lemma \ref{jh} and a standard diagonalization.
\end{proof}

 Hence, assuming  in the above lemma that $(\delta_n)_n$ is a null sequence, we get that  for
every $l\in\nn$  and every sequence
$\big((s_j^n)_{j=1}^l\big)_{n}$ of plegma $l-$tuples in
$\ff\upharpoonright M$ with $s_1^n(1)\to \infty$ the sequence
$(\|\sum_{j=1}^la_jx_{s_j}^n\|)_{n}$ is convergent   and  its
limit is independent from the choice of $((s_j^n)_{j=1}^l)_{n}$.
Actually, we may define a  seminorm $\|\cdot\|_*$ on $c_{00}(\nn)$
under which the natural Hamel basis $(e_n)_n$ satisfies
\begin{equation}\Bigg{|}\Big\|\sum_{j=1}^k
a_jx_{s_j}\Big\|-\Big\|\sum_{j=1}^k
a_je_j\Big\|_*\Bigg{|}\leq\delta_l,\end{equation} for all $1\leq
k\leq l$, $(a_i)_{i=1}^k$  in $[-1,1]$ and
$(s_j)_{j=1}^k\in\text{{Plm}}_k(\ff\upharpoonright M)$ with
$s_1(1)\geq M(l)$.

  Let's notice here that there  exist bounded
  $\ff$-sequences in Banach spaces such that no seminorm resulting
  from Lemma \ref{B-Sp} is a norm. For example this happens in the  case where
  $(x_s)_{s\in\ff}$ is constant.
  Moreover, even if the  $\|\cdot\|_* $ is a norm on
  $c_{00}(\nn)$, the sequence $(e_n)_{n}$ is not necessarily Schauder
  basic.

  We are now ready to give the definition of the $\ff$-spreading
  models of a Banach space $X$.

\begin{defn}\label{Definition of spreading model}
   Let   $X$ be a Banach space, $\ff$ be a regular thin family,
   $(x_s)_{s\in\ff}$ be an $\ff$-sequence in $X$. Let $(E,\|\cdot\|_*)$ be an infinite dimensional
   seminormed linear space with Hamel basis
   $(e_n)_{n}$. Also let  and $M\in[\nn]^\infty$ and $(\delta_n)_n$
   be a null sequence of positive real numbers.

  We will say that the $\ff$-subsequence
  $(x_s)_{s\in\ff\upharpoonright M}$ generates  $(e_n)_{n}$ as an
  $\ff$-spreading model \big(with respect to $(\delta_n)_n$\big) if for every $l\in \nn$, $1\leq k\leq l$, $(a_i)_{i=1}^k$  in $[-1,1]$ and
$(s_j)_{j=1}^k\in\text{{Plm}}_k(\ff\upharpoonright M)$
  with $s_1(1)\geq M(l)$, we have
  \begin{equation}\Bigg{|}\Big{\|}\sum_{j=1}^k a_j x_{s_j}\Big{\|}-\Big{\|}\sum_{j=1}^k a_j
  e_j\Big{\|}_* \Bigg{|}\leq\delta_l.\end{equation}
We will also say that $(x_s)_{s\in\ff}$ admits $(e_n)_{n}$ as an
$\ff$-spreading model
   if there exists  $M\in[\nn]^\infty$ such that  $(x_s)_{s\in\ff\upharpoonright M}$ generates  $(e_n)_{n}$ as an
  $\ff$-spreading model.

  Finally, for a subset $A$ of $X$, we will say that $(e_n)_{n}$ is
  an $\ff$-spreading model of $A$ if there exists an $\ff$-sequence $(x_s)_{s\in\ff}$ in $A$ which admits $(e_n)_{n}$ as
  an $\ff$-spreading model.
\end{defn}
The next remark is straightforward.
\begin{rem}\label{remark on the definition of spreading model}
Let $M\in[\nn]^\infty$ such that $(x_s)_{s\in\ff\upharpoonright
M}$ generates $(e_n)_{n}$ as an $\ff$-spreading model. Then the
following are satisfied.
  \begin{enumerate}
    \item[(i)] The sequence $(e_n)_{n}$ is spreading, i.e. for every $n\in\nn$, $k_1<\ldots<k_n$ in
$\nn$ and $a_1,\ldots,a_n\in\rr$ we have that $\|\sum_{j=1}^na_j
e_j\|_*=\|\sum_{j=1}^n a_j e_{k_j}\|_*$.
    \item[(ii)] For every $M'\in[M]^\infty$ we have that $(x_s)_{s\in\ff\upharpoonright M'}$ generates $(e_n)_{n}$ as an $\ff$-spreading model.
    \item[(iii)] For every  null sequence $(\delta'_n)_{n}$ of positive reals there exists $M'\in[M]^\infty$ such that
    $(x_s)_{s\in\ff\upharpoonright M'}$ generates $(e_n)_{n}$ as an $\ff$-spreading model with respect to $(\delta'_n)_{n}$.
  \end{enumerate}
\end{rem}
By Lemma \ref{B-Sp} we get  the following.
\begin{thm}
Let $\ff$ be a regular thin family and $X$ a Banach space. Then
every bounded $\ff$-sequence in $X$ admits an $\ff$-spreading
model. In particular for every bounded $\ff$-sequence
$(x_s)_{s\in\ff}$ in $X$ and every $N\in[\nn]^\infty$ there exists
$M\in[N]^\infty$ such that $(x_s)_{s\in\ff\upharpoonright M}$
generates $\ff$-spreading model.
\end{thm}

  \subsection{Spreading models of order $\xi$}
  In this subsection we show that Definition \ref{Definition of spreading
  model} is independent of the particular regular thin family
  $\ff$ and actually depends on the order of $\ff$. More precisely we have the following.
\begin{lem}\label{propxi}
  Let $\ff,\g$ be regular thin families with $o(\ff)\leq o(\g)$. Let $X$ be a
  Banach space and $(x_s)_{s\in\ff}$ be an
  $\ff$-sequence in $X$ which admits an
  $\ff$-spreading model $(e_n)_{n}$. Then there exists a $\g$-sequence
  $(w_t)_{t\in\g}$ with $\{w_t:t\in\g\}\subseteq\{x_s:s\in\ff\}$ and which admits  $(e_n)_{n}$ as a $\g$-spreading model.
\end{lem}
\begin{proof}
  Let $M\in[\nn]^\infty$ and $(\delta_n)\searrow0$ be such that $(x_s)_{s\in\ff\upharpoonright M}$ generates $(e_n)_n$ as an $\ff$-spreading model
  with respect to $(\delta_n)_n$.
  By Theorem \ref{plegma preserving maps} there exist $N\in[\nn]^\infty$ and
  a plegma preserving map
  $\varphi:\g\upharpoonright N\to\ff\upharpoonright M$ such that $\min\varphi(t)\geq M(l)$, for every
  $l\in\nn$ and $t\in\g\upharpoonright N$ with $\min t\geq N(l)$.
  For
  every $t\in\g\upharpoonright N$ let  $w_t=x_{\varphi(t)}$ and for every $t\in\g\setminus(\g\upharpoonright N)$
  let $w_t=x_{s_0}$ where $s_0$ is  an arbitrary element of $\ff$. We claim that $(w_t)_{t\in\g\upharpoonright N}$
  generates $(e_n)_{n}$ as a $\g$-spreading model with respect to $(\delta_n)_n$.

  Indeed, fix  $l\in\nn$, $1\leq k\leq l$, $(a_j)_{j=1}^k$ in
  $[0,1]$ and $(t_j)_{j=1}^k\in\g\upharpoonright N$ with $t_1(1)\geq
  N(l)$.  Let
  $s_j=\varphi(t_j)$, for all $1\leq j\leq k$. Then $(s_j)_{j=1}^k\in\text{Plm}_l(\ff\upharpoonright M)$
  and $s_1(1)\geq M(l)$. Therefore, since $(x_s)_{s\in\ff\upharpoonright M}$ generates $(e_n)_n$ as an $\ff$-spreading model
  with respect to $(\delta_n)_n$, we have
\[\Bigg{|}\Big{\|}\sum_{j=1}^k a_j
w_{t_j}\Big{\|}-\Big{\|}\sum_{j=1}^k a_j
  e_j\Big{\|}_* \Bigg{|} =\Bigg{|}\Big{\|}\sum_{j=1}^k a_j x_{s_j}\Big{\|}-\Big{\|}\sum_{j=1}^k a_j
  e_j\Big{\|}_* \Bigg{|}\leq\delta_l\]
  and the proof is complete.
\end{proof}

\begin{cor}\label{equalspr}
  Let $X$ be a Banach space, $A\subseteq X$ and $\ff,\g$ be regular thin families with $o(\ff)=o(\g)$.
  Then
  $(e_n)_{n}$ is an $\ff$-spreading model of $A$ iff $(e_n)_{n}$
  is a $\g$-spreading model of $A$.
\end{cor}
The above permits us to give the following definition.
\begin{defn}
  Let $A$ be a subset of a Banach space $X$ and $\xi\geq 1$
  be a countable ordinal. We will say that $(e_n)_{n}$ is a $\xi$-spreading model of $A$ if there
  exists a regular thin family $\ff$ with $o(\ff)=\xi$ such that
  $(e_n)_{n}$ is an $\ff$-spreading model of $A$. The set of all $\xi$-spreading models of $A$ will be denoted by
  $\mathcal{SM}_\xi(A)$.
\end{defn}
 Notice that by Lemma \ref{propxi} we have \begin{equation}\mathcal{SM}_\zeta(A)\subseteq \mathcal{SM}_\xi(A),\end{equation} for every
$1\leq\zeta<\xi<\omega_1$.

The following  is an  extension of Example 1  in  \cite{AKT}. It
shows that for a given  $\xi<\omega_1$  and a regular thin family
$\g$ there exists a norm on $c_{00}(\g)$ such that setting
$A=\{e_s:s\in\g\}$ (where $(e_s)_{s\in\ff}$ is   the natural Hamel
basis of $c_{00}(\g)$), we have $\mathcal{SM}_\zeta(A)\subsetneq
\mathcal{SM}_\xi(A)$, for every $\zeta<\xi$.\\

\begin{examp}  Let  $1\leq \xi<\omega_1$, $\g$ be a regular thin family of
order $\xi$ and  $(e_s)_{s\in\g}$ be  the natural Hamel basis of
$c_{00}(\g)$. Let $(E,\|\cdot\|)$ be a Banach space with a
normalized spreading and $1$-unconditional
 basis $(e_n)_{n}$ which in addition is not
equivalent to the usual basis of $c_0$. Let $X_\g$ be the completion
of $c_{00}(\g)$ under the norm $\|\cdot\|_\ff$ defined by
\begin{equation}\label{nut}\|x\|_\g=\sup\Big{\{}\Big\|\sum_{i=1}^la_{t_i}e_i\Big\|:l\in\nn,(t_i)_{i=1}^l\in\text{{Plm}}_l(\g)\text{
and }l\leq t_1(1)\Big{\}},\end{equation}
 for every
$x=\sum_{t\in\g} a_t e_t\in c_{00}(\g)$.

Let $A=\{e_t:t\in\g\}$. It is easy to see that $(e_t)_{t\in\g}$
generates $(e_n)_n$ as an $\g$-spreading model, i.e. $(e_n)_{n}$
belongs to $\mathcal{SM}_\xi(A)$. Let $\zeta<\xi$. We claim that for
every  $(e'_n)_{n}\in\mathcal{SM}_\zeta(A)$ either $(e'_n)_n$ is
generated by a constant $\ff$-sequence with $o(\ff)=\zeta$ or it is
isometric to the usual basis of $c_0$. Thus
$(e_n)_n\notin\mathcal{SM}_\zeta(A)$.

Indeed, let  $(e'_n)_{n}\in\mathcal{SM}_\zeta(A)$. Then there exists
a regular thin family $\ff$ of order $\zeta$, an $\ff$-sequence
$(x_s)_{s\in\ff}$ in $A$ and $M\in[\nn]^\infty$ such that
$(x_s)_{s\in\ff\upharpoonright M}$ generates $(e'_n)_{n}$ as an
$\ff$-spreading model. Since $\{x_s\}_{s\in\ff}\subseteq  A$, we may
define $\varphi:\ff\upharpoonright M\to\g$ by choosing for each
$s\in\ff\upharpoonright M$ an element $\varphi(s)\in\mathcal{\g}$
satisfying $e_{\varphi(s)}=x_s$.

 Assume for the following that
$(e'_n)_{n}\in\mathcal{SM}_\zeta(A)$  is not generated by a
constant $\ff$-sequence with $o(\ff)=\zeta$. By part (ii) of
Remark \ref{remark on the definition of spreading model} we have
that $\varphi$ is hereditarily nonconstant and  by Lemma
\ref{lemma making a hereditary nonconstant function, nonconstant
on plegma pairs} there exists $N\in [M]^\infty$ such that for
every plegma pair $(s_1,s_2)$ in $\ff\upharpoonright N$,
$\varphi(s_1)\neq\varphi(s_2)$. Moreover  since $o(\ff)<o(\g)$ by
Theorem \ref{non plegma preserving maps} we have that  there
exists $L\in[N]^\infty$ such that for every plegma pair
$(s_1,s_2)$ in $\g\upharpoonright L$ neither
$(\varphi(s_1),\varphi(s_2))$, nor $(\varphi(s_2),\varphi(s_1))$
is a plegma pair. Therefore, by part (i) of Proposition
\ref{newq1}, we conclude that for every $1\leq k\leq l$,
$(s_j)_{j=1}^k\in\text{{Plm}}_k(\ff\upharpoonright L)$ and
$(t_i)_{i=1}^l\in\text{{Plm}}_l(\g)$, we must have
\begin{equation}\label{uy}\big|\big\{j\in\{1,\ldots,k\}:\varphi(s_j)\in\{t_i:1\leq i\leq
l\}\big\}\big|\leq1.\end{equation} Hence, for every  $k\in\nn$,
$a_1,\ldots,a_k\in\rr$ and
$(s_j)_{j=1}^k\in\text{{Plm}}_k(\ff\upharpoonright L)$, we have
\begin{equation}\Big\|\sum_{j=1}^ka_jx_{s_j}\Big\|_\g=\Big\|\sum_{j=1}^ka_je_{\varphi(s_j)}\Big\|_\g\stackrel{(\ref{nut}),(\ref{uy})}{=}\max_{1\leq
j\leq k}|a_j|,\end{equation} i.e.  $(e'_n)_{n}$ is isometric to
the usual basis of $c_0$.\end{examp}

 A natural question
arising from the above is the following.

\bigskip

\noindent\textit{\textbf{Question.} Let $X$ be a separable Banach
space. Is it true that there exists a countable ordinal $\xi$ such
that $\mathcal{SM}_\zeta(X)=\mathcal{SM}_\xi(X)$ for every
$\zeta>\xi$}?

\bigskip

The above question can be also stated in an isomorphic version, i.e.
whether every sequence in $\mathcal{SM}_\zeta(X)$ is equivalent to
some sequence in $\mathcal{SM}_\xi(X)$  and vice versa.

\begin{rem} In a forthcoming paper we will  provide examples establishing the hierarchy
of the higher order spreading models and also illustrating the
boundaries of the theory. Specifically we will show the following.
\begin{enumerate}
\item[(1)] For every countable limit ordinal $\xi$ there exist a
Banach space $X$ such that $\mathcal{SM}_\xi(X)$ properly includes
up to equivalence $\cup_{\zeta<\xi}\mathcal{SM}_\zeta(X)$.
\item[(2)]  There exist a  Banach space $X$
such that for every $\xi<\omega_1$  and every
$(e_n)_{n}\in\mathcal{SM}_\xi(X)$, the space $E$ generated by
$(e_n)_n$ does not contain any isomorphic copy of $c_0$ or
$\ell^p$ for all $1\leq p<\infty$.
\end{enumerate}
The above results require a deeper study of the structure of
$\ff$-sequences generating $\ell_1$-spreading models (see also
\cite{AKT} for the finite order case).
\end{rem}

\section{$\ff$-sequences in topological spaces}
Let $(X,\ttt)$ be a topological space and $\ff$ be a regular thin
family.  As we have already defined in the previous section, an
$\ff$-sequence in $X$ is any map of the form $\varphi:\ff\to X$ and
generally an $\ff$-subsequence in $X$ is any map of the form
$\varphi:\ff\upharpoonright M\to X$. In this section we will study
the topological properties of $\ff$-sequences. The particular case
where $\ff= [\nn]^k$, $k\in\nn$ had been studied in \cite{AKT}.

\subsection{Convergence of $\ff$-sequences} We introduce the
following natural definition of convergence of $\ff$-sequences.
\begin{defn}\label{defn convergence of f-sequences}
  Let $(X,\ttt)$ be a topological space, $\ff$ a regular thin family, $M\in[\nn]^\infty$, $x_0\in X$ and
  $(x_s)_{s\in\ff}$ an $\ff$-sequence in $X$. We will say that the
  $\ff$-subsequence $(x_s)_{s\in\ff\upharpoonright M}$ converges
  to $x_0$ if for every $U\in\ttt$ with $x_0\in U$ there exists
  $m\in \nn$ such that for every $s\in \ff\upharpoonright M$ with
  $\min s\geq M(m)$ we have that $x_s\in U$.
\end{defn}
 It is immediate that if an $\ff$-subsequence
  $(x_s)_{s\in\ff\upharpoonright M}$ in a topological space $X$ is
  convergent to some $x_0$, then every further $\ff$-subsequence
  is also convergent to $x_0$.
  Also notice that if
  $o(\ff)\geq 2$ then  the convergence of
  $(x_s)_{s\in\ff\upharpoonright M}$ does not in general imply that $\{x_s:s\in\ff\upharpoonright M\}$
   is a relatively compact subset  of $X$. For instance, let
  $(x_s)_{s\in [\nn]^2}$ be the $[\nn]^2$-sequence in $c_0$ defined by  $x_s=\sum_{i=s(1)}^{s(2)}e_i$,
   where $(e_i)_i$ is the usual basis of $c_0$.
  According to Definition  \ref{defn convergence of f-sequences} the
  $[\nn]^2$-sequence
  $(x_s)_{s\in [\nn]^2}$ weakly converges to zero   but
  $\overline{\{x_s :s\in [\nn]^2\}}^{w}=\{x_s: s \in [\nn]^2\}\cup \{0\}$
  which is not a weakly compact subset of  $c_0$.

\begin{prop}\label{arxi metaforas gia ff sequences}
  Let $(X,\ttt_X)$, $(Y,\ttt_Y)$ be two topological spaces and
  $f:Y\to X$ be a continuous map. Let $\ff$ be a regular thin
  family, $M\in[\nn]^\infty$ and $(y_s)_{s\in\ff}$ an
  $\ff$-sequence in $Y$. Suppose that the $\ff$-subsequence $(y_s)_{s\in\ff\upharpoonright
  M}$ is convergent to some $y\in Y$. Then the $\ff$-subsequence $(f(y_s))_{s\in\ff\upharpoonright
  M}$ is convergent to $f(y)$.
\end{prop}
\begin{proof}
  Let $U_X\in\ttt_X$, with $f(y)\in U_X$. By the continuity of $f$
  there exists $U_Y\in\ttt_Y$ such that $y\in U_Y$ and $f[U_Y]\subseteq
  U_X$. Since $(y_s)_{s\in\ff\upharpoonright
  M}$ is convergent to $y$, there exists $m\in\nn$ such that for every $s\in \ff\upharpoonright M$ with
  $\min s\geq M(m)$ we have that $y_s\in U_Y$ and therefore $f(y_s)\in f[U_Y]\subseteq
  U_X$.
\end{proof}

For the rest of this section we shall restrict to $\ff$-sequences in
metric spaces.
\begin{defn}\label{defn Chauchy of f-sequences}
  Let $(X,\rho)$ be a metric space, $\ff$ a regular thin family, $M\in[\nn]^\infty$ and
  $(x_s)_{s\in\ff}$ an $\ff$-sequence in $(X,\rho)$. We will say that the
  $\ff$-subsequence $(x_s)_{s\in\ff\upharpoonright M}$ is Cauchy
  if for every $\ee>0$ there exists $m\in\nn$ such that for every
  $s_1,s_2\in\ff\upharpoonright M$ with $\min s_1,\min s_2\geq
  M(m)$, we have that $\rho(x_{s_1},x_{s_2})<\ee$.
\end{defn}

\begin{prop}\label{remark on ff Cauchy-convergent sequences and their spreading model}
  Let $M\in[\nn]^\infty$,  $\ff$ be a regular thin family and  $(x_s)_{s\in\ff}$ be an $\ff$-sequence
  in a complete metric space $(X,\rho)$. Then  the $\ff$-subsequence $(x_s)_{s\in\ff\upharpoonright M}$ is Cauchy
  if and only if $(x_s)_{s\in\ff\upharpoonright M}$ is  convergent.
\end{prop}
\begin{proof}
  If the $\ff$-subsequence $(x_s)_{s\in\ff\upharpoonright M}$ is  convergent, then it is straightforward that $(x_s)_{s\in\ff\upharpoonright M}$
  is Cauchy. Concerning the converse we have the following. Suppose that the $\ff$-subsequence $(x_s)_{s\in\ff\upharpoonright M}$
  is  Cauchy. Let $(s_n)_{n}$ be a sequence in $\ff\upharpoonright M$ such that $\min s_n\to \infty$.
  It is immediate that $(x_{s_n})_{n}$ forms a  Cauchy sequence in $X$. Since $(X, \rho)$ is complete, there exists $x\in  X$ such that
  the sequence $(x_{s_n})_{n}$ converges to $x$. We will show that the $\ff$-subsequence $(x_s)_{s\in\ff\upharpoonright M}$ converges to $x$.
  Indeed, let $\varepsilon>0$. Since $(x_s)_{s\in\ff\upharpoonright M}$
  is Cauchy, there exists $k_0\in\nn$ such that for every $t_1,t_2\in\ff \upharpoonright M$ with $\min t_1,\min t_2\geq M(k_0)$ we have that
  $\rho(x_{t_1},x_{t_2})<\frac{\varepsilon}{2}$. Since the sequence $(x_{s_n})_{n}$ converges to $x$ and $\min s_n\to \infty$,
  there exists $n_0\in\nn$ such that $\min s_{n_0}\geq M(k_0)$ and $\rho(x,x_{s_{n_0}})<\varepsilon/2$. Hence for every $s\in\ff\upharpoonright M$
  such that $\min s\geq M(k_0)$, we have that
  $\rho(x,x_s)\leq \rho(x, x_{s_{n_0}})+\rho(x_{s_{n_0}},
  x_s)<\varepsilon$
  and the proof is completed.
\end{proof}

\begin{lem}\label{Lemma epsilon near admissble yield Cauchy}
  Let $M\in[\nn]^\infty$,  $\ff$ be a regular thin family and $(x_s)_{s\in\ff}$ an $\ff$-sequence
  in a metric space $(X, \rho)$. Suppose that for every $\varepsilon>0$ and  $L\in[M]^\infty$ there exists a plegma pair
  $(s_1,s_2)$ in $\ff\upharpoonright L$ such that  $\rho(x_{s_1}, x_{s_2})<\varepsilon$. Then the  $\ff$-subsequence
   $(x_s)_{s\in\ff\upharpoonright M}$ has a  further Cauchy subsequence.
\end{lem}
\begin{proof}
  Let $(\varepsilon_n)_{n}$ be a sequence of positive reals such that $\sum_{n=1}^\infty \varepsilon_n<\infty$. Using Theorem
  \ref{ramseyforplegma}, we inductively
  construct a decreasing sequence $(L_n)_{n}$ in $[M]^\infty$, such that for every $n\in\nn$ and for every plegma
  pair $(s_1,s_2)$ in $\ff\upharpoonright L_n$ we have that $\rho(x_{s_1}, x_{s_2})<\varepsilon_n$. Let $L'$ be
  a diagonalization of $(L_n)_{n}$, i.e. $L'(n)\in L_n$ for all $n\in\nn$, and $L=\{L'(2n):n\in\nn\}$.

  We claim that the $\ff$-subsequence
  $(x_s)_{s\in\ff\upharpoonright L}$ is Cauchy. Indeed let $\varepsilon>0$. There exists $n_0\in \nn$ such that
  $\sum_{n=n_0}^\infty\varepsilon_n<\frac{\varepsilon}{2}$. Let $s_0$ be the unique initial segment of $\{L(n):n\geq n_0\}$ in $\ff$.
  If $\max s_0=L(k)$ then we set $k_0=k+1$. Then for every $s_1,s_2\in\ff\upharpoonright L$ with $\min s_1, \min s_2\geq L(k_0)$,
  by Theorem \ref{accessing everything with plegma path of length |s_0|} there exist plegma paths $(s_j^1)_{j=1}^{|s_0|},
  (s_j^2)_{j=1}^{|s_0|}$ in $\ff\upharpoonright L'$ from $s_0$ to $s_1,s_2$ respectively. Then for $i=1,2$ we have that
  \[\rho(x_{s_0}, x_{s_i})\leq\sum_{j=0}^{|s_0|-1}\rho(x_{s_j^i},x_{s_{j+1}^i})<\sum_{j=0}^{|s_0|-1}\varepsilon_{n_0+j}<\frac{\varepsilon}{2}\]
  which implies that $\rho(x_{s_1}, x_{s_2})<\varepsilon$.
\end{proof}
\begin{defn} Let $\varepsilon>0$, $L\in[\nn]^\infty$,  $\ff$ be a regular thin family and $(x_s)_{s\in\ff}$ an $\ff$-sequence
   in a metric space $X$. We will say that the subsequence $(x_s)_{s\in\ff\upharpoonright L}$ is plegma $\varepsilon$-separated
    if for every plegma pair $(s_1,s_2)$ in $\ff\upharpoonright L$,  $\rho(x_{s_1}, x_{s_2})>\varepsilon$.
\end{defn}
The following proposition is actually a restatement of Lemma
\ref{Lemma epsilon near admissble yield Cauchy}.
\begin{prop}\label{Proposition equivalence of non Cauchy and separability}
  Let $M\in[\nn]^\infty$,  $\ff$ be a regular thin family and $(x_s)_{s\in\ff}$ an $\ff$-sequence
  in a metric space $X$. Then the  following are equivalent.
  \begin{enumerate}
    \item [(i)] The $\ff$-subsequence $(x_s)_{s\in\ff\upharpoonright M}$ has no  further Cauchy subsequence.
    \item [(ii)] For every $N\in [M]^\infty$ there exist $\varepsilon>0$ and $L\in [N]^\infty$  such that the
  subsequence $(x_s)_{s\in\ff\upharpoonright L}$ is plegma $\varepsilon$-separated.
  \end{enumerate}
\end{prop}
\begin{proof}
  (i)$\Rightarrow$(ii): Assume that (ii) is not true. Then there is
  $N\in [M]^\infty$ such that for every $\varepsilon>0$ and  $L\in[N]^\infty$ there exists a plegma pair
  $(s_1,s_2)$ in $\ff\upharpoonright L$ such that  $\rho(x_{s_1},
  x_{s_2})<\varepsilon$. By Lemma \ref{Lemma epsilon near admissble yield
  Cauchy} the $\ff$-subsequence
   $(x_s)_{s\in\ff\upharpoonright N}$ has a  further Cauchy subsequence.
    Since $N\subseteq M$ this means that $(x_s)_{s\in\ff\upharpoonright M}$ has a  further Cauchy subsequence
    which is   a contradiction.

(ii)$\Rightarrow $(i): Suppose that (i) does not hold. Then there
exists $N\in [M]^\infty$ such that $(x_s)_{s\in\ff\upharpoonright
N}$ is Cauchy. Let $\varepsilon>0$ and $L\in [N]^\infty$. Then
$(x_s)_{s\in\ff\upharpoonright L}$ is also Cauchy and therefore
$(x_s)_{s\in\ff\upharpoonright L}$ is not   plegma
$\varepsilon$-separated, a contradiction.
\end{proof}

\subsection{Subordinated $\ff$-sequences}
By identifying every subset of $\nn$ with its characteristic
function, a  thin family  $\ff$ becomes a discrete subspace of
$\{0,1\}^\nn$ (under the usual product topology) with
$\widehat{\ff}$ being  its closure. This in particular yields that
every $\phi:\ff\to (X,\ttt)$ is automatically continuous. In this
subsection we show  that for every regular thin family $\ff$, $M\in
[\nn]^\infty$ and $\varphi:\ff\upharpoonright M\to (X,\mathcal{T})$
such that the closure of $\varphi(\ff\upharpoonright M)$ is a
compact metrizable subspace of $X$  there exist $L\in[M]^\infty$ and
a continuous extension
$\widehat{\varphi}:\widehat{\ff}\upharpoonright
L\to(X,\mathcal{T})$. We start with the following definition.
  \begin{defn}\label{defsub}
    Let $(X,\ttt)$ be a topological space, $\ff$ be a regular thin family, $M\in[\nn]^\infty$ and
  $(x_s)_{s\in\ff}$ be an $\ff$-sequence in $X$. We say
    that $(x_s)_{s\in\ff\upharpoonright M}$ is subordinated (with respect to $(X,\mathcal{T})$)
     if there exists a continuous map $\widehat{\varphi}:\widehat{\ff}\upharpoonright M\to (X,\mathcal{T})$
with  $\widehat{\varphi}(s)=x_s$ for every  $s\in\ff\upharpoonright
M$.

 \end{defn}
 Assume that  $(x_s)_{s\in\ff\upharpoonright M}$ is subordinated. Then since
 $\ff$ is dense in $\widehat{\ff}$, there exists a unique  continuous map
     $\widehat{\varphi}:\widehat{\ff}\upharpoonright M\to (X,\mathcal{T})$ witnessing
     this. Moreover, for the same reason we have
    $\overline{\{x_s:s\in\ff\upharpoonright M\}}=\widehat{\varphi}\big(\widehat{\ff}\upharpoonright
    M\big)$,
    where $\overline{\{x_s:s\in\ff\upharpoonright M\}}$ is the $\mathcal{T}$-closure of
    $\{x_s : s\in\ff\upharpoonright M\}$ in $X$. Therefore
    $\overline{\{x_s:s\in\ff\upharpoonright M\}}$ is a countable compact metrizable subspace of $(X,\ttt)$
     with Cantor-Bendixson index at most $o(\ff)+1$. Another property of subordinated $\ff$-sequences
     is stated in the next   proposition.

\begin{prop}\label{subordinating yields convergence}
  Let $(X,\ttt)$ be a topological space, $\ff$ be  a regular thin family and
  $(x_s)_{s\in\ff}$ be an $\ff$-sequence in $X$. Let $M\in[\nn]^\infty$ such that  $(x_s)_{s\in\ff\upharpoonright M}$ is
  subordinated. Then $(x_s)_{s\in\ff\upharpoonright
  M}$ is a convergent $\ff$-subsequence in $X$. In particular, if $\widehat{\varphi}:\widehat{\ff}\upharpoonright M\to
  (X,\mathcal{T})$ is the continuous map witnessing the fact that $(x_s)_{s\in\ff\upharpoonright M}$ is
  subordinated then  $(x_s)_{s\in\ff\upharpoonright
  M}$ is convergent to $\widehat{\varphi}(\varnothing)$.
\end{prop}
\begin{proof} Via the identity map we may consider the family $\ff$ as an $\ff$-sequence in the
metric space $Y=\widehat{\ff}$ i.e.  let $(y_s)_{s\in \ff}$ be the
$\ff$-sequence in $Y$, with $y_s=s$ for every $s\in \ff$.  As we
have already noticed $(y_s)_{s\in \ff}$ converges to the empty set.
Hence, since $\widehat{\varphi}:\widehat{\ff}\upharpoonright M\to
(X,\mathcal{T})$ is continuous, by Proposition \ref{arxi metaforas
gia ff sequences} we get that
$\big(\widehat{\varphi}(y_s)\big)_{s\in \ff\upharpoonright M}$
converges to $\widehat{\varphi}(\varnothing)$. Since
$\widehat{\varphi}(y_s)=\widehat{\varphi}(s)=x_s$ for every $s\in
\ff\upharpoonright M$, this means that
$(x_s)_{s\in\ff\upharpoonright M}$ is convergent to
$\widehat{\varphi}(\varnothing)$.
\end{proof}
\begin{thm}\label{Create subordinated}
Let $\ff$ be  a  regular thin family and $(x_s)_{s\in\ff}$ be an
$\ff$-sequence in a topological space $(X,\mathcal{T})$. Then for
every $M\in[\nn]^\infty$ such that $\overline{\{x_s:
s\in\ff\upharpoonright M\}}$  is a compact metrizable subspace of
$(X,\ttt)$ there exists $L\in[M]^\infty$ such that
$(x_s)_{s\in\ff\upharpoonright L}$ is subordinated.
\end{thm}
\begin{proof}
We will use induction on the order of the regular thin family $\ff$.
If $o(\ff)=0$ (i.e. the family $\ff$ is the singleton
$\ff=\{\varnothing\}$) the result trivially holds. Let
$\xi<\omega_1$ and assume that  the theorem is true when
$o(\ff)<\xi$.

 We fix  a regular thin family  $\ff$ with $o(\ff)=\xi$,
an $\ff$-sequence $(x_s)_{s\in\ff}$   in a topological space
$(X,\ttt)$ and $M\in[\nn]^\infty$ such that $\overline{\{x_s:
s\in\ff\upharpoonright M\}}$  is a compact metrizable subspace of
$(X,\ttt)$. By passing to an infinite subset of $M$ if it is
necessary we may also suppose that $\ff$ is very large in $M$. Let
$\rho$ be a compatible metric for the subspace $X_0=\overline{\{x_s:
s\in\ff\upharpoonright M\}}$. We shall construct (a) a strictly
increasing sequence $(m_n)_n$ in $M$, (b) a decreasing sequence
$M=M_0\supseteq M_1\supseteq ... $ of infinite subsets of $M$ (c) a
sequence $\widehat{\varphi}_n$ of maps with
$\widehat{\varphi}_n:\widehat{\ff_{(m_n)}}\upharpoonright M_n\to X$
and (d) a decreasing sequence of closed balls $(B_n)_n$ in $X_0$
such that for every $n\in\nn$ the following are satisfied:
\begin{enumerate}
\item[(i)] $m_n=\min M_{n-1}$ and $M_n\subseteq
M_{n-1}\setminus\{m_n\}$,
\item[(ii)] $\text{diam}\  B_n<1/n$,
\item[(iii)]  the map  $\widehat{\varphi}_n$ is continuous,
\item[(iv)] $\widehat{\varphi}_n (u)=x_{\{m_n\}\cup u}$, for every
$u\in \ff_{(m_n)}\upharpoonright M_n$ and
\item[(v)] $\{\widehat{\varphi}_n (u):
u\in\widehat{\ff_{(m_n)}}\upharpoonright M_n\}\subseteq B_n$.
\end{enumerate}
We shall present the general inductive step of the above
construction so let us  assume that the construction has been
carried out up to some $n\in\nn$. We set $m_{n+1}=\min M_n$. Since
$\ff$ is very large in $M$ we have that $\g=\ff_{(m_{n+1})}=\{u\in
[\nn]^{<\infty}:m_{n+1}<u \ \text{and} \ \{m_{n+1}\}\cup u\in\ff\}$
is  a regular thin family. For each $u\in \g$ we set
$y_u=x_{\{m_{n+1}\}\cup u}$ and we form the $\g$-sequence
$(y_u)_{u\in\g}$. Let $M'_n=M_n\setminus\{m_{n+1}\}$. Since
$Y=\{y_u: u\in\g\upharpoonright M'_n\}\subseteq \overline{\{x_s:
s\in\ff\upharpoonright M\}}$, we have that the closure of $Y$ in
$(X,\ttt)$  is also a compact metrizable subspace of $(X,\ttt)$.
Thus $Y$ is a totally bounded metric space and therefore by Theorem
\ref{Galvin prikry} and passing to an infinite subset of $M'_n$ if
it is necessary, we may also suppose that there exists a ball
$B_{n+1}$ of $X_0$ with $\text{diam} \ B_{n+1}<(n+1)^{-1}$ and such
that \begin{equation}\label{ball}\{y_u: u\in\g\upharpoonright
M'_n\}\subseteq B_{n+1}.\end{equation} Moreover,
$o(\g)=o\big(\ff_{(m_{n+1})}\big)<o(\ff)=\xi$. Hence, by our
inductive hypothesis, there exists an infinite subset $M_{n+1}$ of
$M'_n=M_n\setminus \{m_{n+1}\}$ such that the $\g$-subsequence
$(y_u)_{u\in\g\upharpoonright M_{n+1}}$ is subordinated. Let
$\widehat{\varphi}_{n+1}:\widehat{\g}\upharpoonright M_{n+1}\to X$
be the continuous  map witnessing this fact. Then
$\widehat{\varphi}_{n+1}(u)=y_u= x_{\{m_{n+1}\}\cup u}$, for every
$u\in\ff_{(m_{n+1})}\upharpoonright M_{n+1}$ and  by the continuity
of $\widehat{\varphi}_{n+1}$ we have
\begin{equation}
\big\{\widehat{\varphi}_{n+1} (u):
u\in\widehat{\ff_{(m_{n+1})}}\upharpoonright
M_{n+1}\big\}\subseteq\overline {\{y_u:
u\in\ff_{(m_{n+1})}\upharpoonright
M_{n+1}\}}\stackrel{(\ref{ball})}{\subseteq} B_{n+1},
\end{equation}
which completes the proof of the inductive step.

We set  $M'=\{m_n:n\in\nn\}$. Since $\lim\text{diam} \ B_n=0$ and
$X_0$ is a compact metric space  there exists a strictly increasing
sequence $(k_n)_n$ and $x_0\in X_0$ such that
\begin{equation}\label{balll}\lim\text{dist} \big(x_0, B_{k_n}\big)= 0.\end{equation} We set
$L=\{m_{k_n}: n\in\nn\}$ and we define
$\widehat{\varphi}:\widehat{\ff}\upharpoonright L\to X$ as follows.
For $s=\varnothing$, we set $\widehat{\varphi}(\varnothing)=x_0$.
Otherwise, if $n$ is the unique positive integer such that
$m_{k_n}=\min s$ we set
$\widehat{\varphi}(s)=\widehat{\varphi}_{k_n}(s\setminus\{\min
s\})=x_t$. It is easy to check that  $\widehat{\varphi}$ is well
defined. To see that $\widehat{\varphi}$ is continuous let $(s_n)_n$
be a sequence in $\widehat{\ff}\upharpoonright L$ and
$s\in\widehat{\ff}\upharpoonright L$ such that $s_n\to s$. If
$s=\varnothing$ we have that $\min s_n\to +\infty$, thus using
condition (v)  and equation (\ref{balll}) we obtain  that
$\widehat{\varphi}(s_n)\to x_0=\widehat{\varphi}(\varnothing)$.
Otherwise, let $m_{k_{n_0}}=\min s$. Then $\min s_n=\min
s=m_{k_{n_0}}$, for all but finitely many $n$. Therefore,
$\widehat{\varphi}(s_n)=\widehat{\varphi}_{k_{n_0}}(s_n)$, for all
but finitely many $n$ and since $\widehat{\varphi}_{k_{n_0}}$ is
continuous,  $\widehat{\varphi}(s_n)\to
\widehat{\varphi}(s)$.\end{proof}
\section{$\ff$-sequences generating spreading models}
Let $(x_n)_n$ be a sequence in a Banach space $X$ generating  a
spreading model $(e_n)_n$. It is well known (see \cite{BL},
\cite{BS}) that if $(x_n)_n$ is norm convergent then the seminorm in
the space $E$ generated by $(e_n)_n$ is not a norm. Furthermore if
$(x_n)_n$ is weakly null and seminormalized then $(e_n)_n$ is
$1$-unconditional. In this section we show that analogues of these
results remain true in the higher order setting of  $\xi$-spreading
models.  We begin with a short review of the basic properties of
spreading sequences.
 \subsection{Spreading sequences}
Let  $(E,\|\cdot\|_*)$  be a seminormed linear space. A sequence
$(e_n)_{n}$ in $E$ is called \emph{spreading} if
$\|\sum_{j=1}^na_je_j\|_*=\|\sum_{j=1}^na_je_{k_j}\|_*$, for every
$n\in\nn$, $a_1,\ldots,a_n\in \rr$ and $k_1<\ldots<k_n$ in $\nn$.
 As we have already mentioned every spreading model
 of any order of a Banach space is a spreading sequence.
 In this subsection we shall briefly  recall some well known results
 on spreading sequences that we shall later  use (for a more detailed exposition see \cite{AK}, \cite{AT}, \cite{BL}, \cite{BS}).
We start with the following elementary lemma.
\begin{lem}\label{sing}
Let $(E,\|\cdot\|_*)$ be a seminormed linear space and $(e_n)_{n}$
be a spreading sequence in $E$. Then the following are equivalent.
\begin{enumerate}
\item[(i)] There exist $k\in\nn$ and $a_1,\ldots,a_k\in\rr$ not
all zero with $\|\sum_{j=1}^ka_ie_i\|_*=0$. \item[(ii)] For
every $n,m\in\nn$, $\|e_n-e_m\|_*=0$.
\end{enumerate}
\end{lem}
\begin{proof} The implication  (ii)$\Rightarrow$(i) is straightforward. To show
the converse let $k\in\nn$ and $a_1,\ldots,a_n\in\rr$ not all zero,
such that $\|\sum_{j=1}^ka_je_j\|_*=0$. Since $(e_n)_{n}$ is
spreading  we may suppose that $a_j\neq 0$ for all $1\leq j\leq  n$.
  Moreover notice that
  \begin{equation}\Big{\|}\sum_{j=1}^{k-1}a_je_j+a_ke_k\Big{\|}_*=\Big{\|}\sum_{j=1}^{k-1}a_je_j+a_ke_{k+1}\Big{\|}_*=0.\end{equation}
Hence,
  \begin{equation}\Big{\|}e_k-e_{k+1}\Big{\|}_*\leq\frac{1}{|a_k|}
  \Bigg{(} \Big{\|}\sum_{j=1}^{k-1}a_je_j+a_ke_k\Big{\|}_*+\Big{\|}\sum_{j=1}^{k-1}a_je_j+a_ne_{k+1}\Big{\|}_* \Bigg{)}=0. \end{equation}
  Since $(e_n)_{n}$ is spreading we get that $\|e_n-e_m\|_*=0$ for every $n,m\in\nn$.
\end{proof}
 Spreading sequences in seminormed linear spaces satisfying
(i) or (ii) of the above lemma will be called \emph{trivial}. As a
consequence  we have that if $(e_n)_{n}$ is non trivial then the
restriction of the seminorm $\|\cdot\|_*$ to the linear subspace
generated by $(e_n)_n$ is actually norm.

As in \cite{AKT}, we classify the non trivial spreading sequences
into the following three categories: (1) the \emph{singular}, i.e.
the non trivial spreading sequences which are not Schauder basic,
(2) the \emph{unconditional} and (3) the \emph{conditional Schauder
basic} spreading sequences, i.e. the spreading sequences which are
Schauder basic but not unconditional.

\begin{prop}\label{equiv forms for 1-subsymmetric weakly null}
Let $(e_n)_{n}$ be a non trivial spreading sequence.
\begin{enumerate}
\item[(i)]  If $(e_n)_n$ is weakly null then it  is
1-unconditional.
\item[(ii)]  If $(e_n)_{n}$ is unconditional then either it is equivalent
to the usual basis of $\ell^1$ or it is norm Ces\`aro summable
to 0 (i.e. $\lim_{n\to\infty}\Big\|\frac{1}{n}\sum_{i=1}^n
e_i\Big\|= 0$).
\end{enumerate}
\end{prop}
\begin{proof}
(i) See \cite{AK}. (ii)  Since $(e_n)_{n}$ is unconditional there
exist $C>0$ such that $\Big\|\sum_{i=1}^n\ee_ia_ie_i\Big\|\leq C
\Big\|\sum_{i=1}^na_ie_i\Big\|$,  for every $n\in\nn$,
$a_1,\ldots,a_n\in\rr$ and  $\ee_1,\ldots,\ee_n\in\{-1,1\}$. Also
since it is spreading and nontrivial there exists $M>0$ such that
$\|e_n\|=M$ for all $n\in\nn$. Suppose that  $(e_n)_{n}$ is not
Ces\`aro summable to zero. Then there exist
  $\theta>0$ and a strictly increasing sequence of natural numbers
  $(p_n)_{n}$ such that
  $\|\frac{1}{p_n}\sum_{i=1}^{p_n}e_i\|>\theta$, for all
  $i\in\{1,\ldots,p_n\}$. Hence for every $n\in\nn$ there exist
  $x^*_n$ with $\|x^*_n\|=1$ such that
  $x^*_n(\frac{1}{p_n}\sum_{i=1}^{p_n}e_i)>\theta$. For every
  $n\in\nn$, we set $I_n=\{1,\ldots,p_n\}$ and let $A_n=\{i\in I_n:\;
  x^*_n(e_i)>\frac{\theta}{2}\}$. Then  we
  have
  \[\begin{split}
    \theta&<x^*_n\Big(\frac{1}{p_n}\sum_{i\in I_n}e_i\Big)
    =\frac{1}{p_n}x^*_n\Big(\sum_{i\in A_n}e_i\Big)
    +\frac{1}{p_n}x^*_n\Big(\sum_{i\in I_n\setminus A_n}e_i\Big)
    \leq\frac{1}{p_n}|A_n|C+\frac{\theta}{2}
  \end{split}\]
  Hence $|A_n|\geq \frac{\theta}{2C}p_n$ which gives that $\lim_{n\to\infty}|A_n|=+\infty$.
  We are now ready to show that $(e_n)_n$ is equivalent to the usual basis of $\ell_1$. Indeed, let
  $n\in\nn$, $a_1,\ldots,a_n\in\rr$ and choose $n_0\in\nn$ such that
  $|A_{n_0}|\geq n$. Then
  \[\begin{split}
     M\sum_{i=1}^n|a_i|\geq\Big\|\sum_{i=1}^n a_i e_i\Big\|&
    \geq \frac{1}{C}\Big\|\sum_{i=1}^n|a_i|e_i\Big\|
    = \frac{1}{C}\Big\|\sum_{i=1}^n|a_i|e_{A_{n_0}(i)}\Big\|\\
    &\geq \frac{1}{C}\cdot
    x^*_n\Big(\sum_{i=1}^n|a_i|e_{A_{n_0}(i)}\Big)
    \geq \frac{\theta}{2C} \sum_{i=1}^n|a_i|.
  \end{split}\]\end{proof}
\begin{prop}\label{properties of the natural decomposition}
Let $(e_n)_{n}$ be a singular sequence and let $E$ be the space
generated by $(e_n)_{n}$. Then there is $e\in E\setminus \{0\}$ such
that $(e_n)_{n}$ is weakly convergent to $e$. Moreover if
$e'_n=e_n-e$ then  $(e'_n)_{n}$ is  spreading, $1$-unconditional and
Ces\`aro summable to zero.
\end{prop}
\begin{proof} Since $(e_n)_{n}$ is equivalent to all its subsequences and
it is not Schauder basic,  every subsequence of $(e_n)_n$ is not
Schauder basic. In particular,  every subsequence of $(e_n)_n$
cannot be non trivial weak-Cauchy or  weakly null. Hence, by
Rosenthal's $\ell^1$-theorem \cite{Ro} $(e_n)_{n}$ is weakly
convergent to a non zero element $e\in E$.

 Let $e'_n=e_n-e$.  To  show that $(e'_n)_{n}$ is spreading, let $n\in\nn$,
$\lambda_1,...,\lambda_n\in\rr$ and $k_1<...<k_n$ in $\nn$. If
$\sum_{i=1}^n\lambda_i=0$, then notice  that
\begin{equation}\label{lk}\Big\|\sum_{i=1}^n\lambda_ie'_i\Big\|=\Big\|\sum_{i=1}^n\lambda_ie_i\Big\|=
\Big\|\sum_{i=1}^n\lambda_ie_{k_i}\Big\|=\Big\|\sum_{i=1}^n\lambda_ie'_{k_i}\Big\|\end{equation}
Generally let $\sum_{i=1}^n\lambda_i=\lambda$. Since $(e'_n)_{n}$ is
weakly null we may choose a convex block subsequence $(w_m)_{m}$ of
$(e'_n)_{n}$ which norm converges to zero. Let $m_0\in\nn$ be such
that $k_n<\text{supp}(w_m)$ for all $m\geq m_0$. Then by (\ref{lk})
 we have
\begin{equation}\Big\|\sum_{i=1}^n\lambda_ie'_i-\lambda
w_m\Big\|=\Big\|\sum_{i=1}^n\lambda_ie'_{k_i}-\lambda
w_m\Big\|,\end{equation}  for all $m\geq m_0$.  Hence, by taking
limits, we get that
$\Big\|\sum_{i=1}^n\lambda_ie'_i\Big\|=\Big\|\sum_{i=1}^n\lambda_ie'_{k_i}\Big\|$,
that is  the sequence  $(e'_n)_{n}$ is spreading. Moreover, since
$\|e_n-e_m\|=\|e'_n-e'_m\|$, by Lemma  \ref{sing} we have that
$(e'_n)_{n}$ is non trivial. Finally, since $(e'_n)_{n}$ is weakly
null, by  Proposition \ref{equiv forms for 1-subsymmetric weakly
null} it is also $1$-unconditional and norm Ces\'aro summable to
zero.\end{proof} In the following the above decomposition
$e_n=e'_n+e$ of a singular spreading sequence $(e_n)_n$ will be
called \emph{the natural decomposition} of $(e_n)_n$.
\subsection{$\ff$-sequences generating non singular spreading
models}
 We start with a characterization
of the $\ff$-sequences in a Banach space $X$ which generate a
trivial spreading model.
\begin{thm}\label{Theorem equivalent forms for having norm on the spreading model}
  Let $X$ be a Banach space,  $\ff$ be a regular thin family and $(x_s)_{s\in\ff}$ be an $\ff$-sequence in  $X$
and $M\in[\nn]^\infty$. Let $(E,\|\cdot\|_*)$ be an infinite
dimensional seminormed linear space with Hamel basis $(e_n)_{n}$
such that $(x_s)_{s\in\ff\upharpoonright M}$ generates $(e_n)_{n}$
as an $\ff$-spreading model. Then the following are equivalent:
\begin{enumerate}
\item[(i)] The sequence $(e_n)_{n}$ is trivial.
\item[(ii)] For every $\varepsilon>0$ and every $L\in[M]^\infty$, there exists a plegma pair $(s_1,s_2)$
 in $\ff\upharpoonright L$ such that
$\|x_{s_1}-x_{s_2}\|<\varepsilon$.
\item[(iii)] The $\ff$-subsequence $(x_s)_{s\in\ff\upharpoonright M}$ contains a further norm Cauchy subsequence.
\item[(iv)] There exists $x\in X$ such that every subsequence of $(x_s)_{s\in\ff\upharpoonright M}$
 contains a further subsequence convergent to $x$.
    \end{enumerate}
\end{thm}
\begin{proof}

  (i)$\Rightarrow$(ii): Let $\varepsilon>0$ and  $L\in[M]^\infty$.
  Since  the $\ff$-subsequence
  $(x_s)_{s\in\ff\upharpoonright L}$ also generates  $(e_n)_{n}$ as an $\ff$-spreading model
   (see Remark \ref{remark on the definition of spreading model}),  there
  exists $n_0\in\nn$ such that for every plegma pair $(s_1,s_2)$ in $\ff\upharpoonright L$ with $\min s_1\geq L(n_0)$, we have
  \begin{equation}\label{klo}\Bigg| \Big\|x_{s_1}-x_{s_2} \Big\|-\Big\| e_1-e_2 \Big\|_*
  \Bigg|<\varepsilon.\end{equation}
  Let $(s_1,s_2)$ be such a plegma pair. Since $(e_n)_n$ is trivial we have $\|e_1-e_2\|_*=0$ and therefore by (\ref{klo}) we obtain that
$\Big\|x_{s_1}-x_{s_2} \Big\|<\varepsilon$.

  (ii)$\Rightarrow$(iii): This follows by Lemma \ref{Lemma epsilon near admissble yield Cauchy}.

  (iii)$\Rightarrow$(i): Using that $(x_s)_{s\in\ff\upharpoonright M}$ contains
  a further norm Cauchy subsequence, we easily construct a sequence $((s_1^n,s_2^n))_{n}$
   of plegma pairs in $\ff\upharpoonright M$ such that
      $s_1^n(1)\to \infty$ and $\|x_{s_1^n}-x_{s_2^n}\|<1/n$. Then
      we have \begin{equation}\big\|e_1-e_2\|_*=\lim_{n\to\infty}\|x_{s_1^n}-x_{s_2^n}\|=0.\end{equation}
    Thus, by Lemma \ref{sing}, we conclude that $(e_n)_n$ is trivial.

   (iv)$\Rightarrow$(iii):  It is straightforward.

   (i)$\Rightarrow$(iv)
      Since  every subsequence of $(x_s)_{s\in\ff \upharpoonright M}$
      generates $(e_n)_{n}$ as an $\ff$-spreading model we have that for every $L\in[M]^\infty$
      $(x_s)_{s\in\ff \upharpoonright M}$ generates a trivial spreading
      model. By the  implication  (i)$\Rightarrow$(iii)  and Proposition
      \ref{remark on ff Cauchy-convergent sequences and their spreading model},
      we have that  every subsequence of $(x_s)_{s\in\ff\upharpoonright M}$
      contains a further convergent subsequence. It remains to show that
      all the convergent subsequences of $(x_s)_{s\in\ff \upharpoonright M}$ have a common limit.

      To this end, let $L_1,L_2\in[M]^\infty$, $x_1,x_2\in X$ such
      that $(x_s)_{s\in\ff \upharpoonright L_i}$ converges to $x_i$ for
      $i\in\{1,2\}$ and
       let $\varepsilon >0$. Hence  there exists $n_0\in\nn$ such that for every
        $s\in\ff\upharpoonright L_1$ and $t\in\ff\upharpoonright L_2$ with $\min s\geq
        L_1(n_0)$ and  $\min t\geq
        L_2(n_0)$
       we have \begin{equation}\label{pre}\|x_1-x_{s}\|, \|x_2-x_{t}\| <\frac{\varepsilon}{3}
       \end{equation}
       Since $(x_s)_{s\in\ff \upharpoonright M}$  generates  the trivial sequence $(e_n)_n$
       as an $\ff$-spreading model, we may  also assume that
       \begin{equation}\label{pree}\big\| x_{s_1}-x_{s_2} \big\|= \Bigg|\Big\| x_{s_1}-x_{s_2} \big\| -\big\|e_1-e_2\Big\|_* \Bigg|
      <\frac{\varepsilon}{3},\end{equation}
        for every plegma pair $(s_1,s_2)$
      in $\ff\upharpoonright M$ with $\min s_1\geq M(n_0)$.

      It is easy to see that we can choose $s_1\in\ff \upharpoonright L_1$ with $\min s_1\geq L_1(n_0)$ and
      $s_2\in\ff\upharpoonright L_2$ with $\min s_2\geq L_2(n_0)$ such that $(s_1,s_2)$ is a plegma pair. Then
      by (\ref{pre}) and (\ref{pree}) we have \begin{equation}\label{preee}\Big\| x_1-x_2 \Big\|\leq \Big\|x_1-x_{s_1}\Big\|+\Big\|
x_{s_1}-x_{s_2} \Big\|+\Big\|x_2-x_{s_2}\Big\|<\varepsilon.\end{equation}
      Since (\ref{preee}) holds for every $\varepsilon>0$ we get  that $x_1=x_2$. The proof is complete.\end{proof}
We proceed to  present a sufficient condition for an $\ff$-sequence
to generate a Schauder basic spreading model. We need the next
definition.
\begin{defn}\label{Def of SSD}
    Let $A$ be a countable seminormalized subset of a Banach space $X$.
    We say that $A$ admits a \textit{Skipped Schauder Decomposition} (SSD) if there exist $C\geq 1$ and
    a pairwise disjoint sequence $(A_k)_{k}$ of finite subsets of $A$ such
    that the following are satisfied.
    \begin{enumerate}
      \item[(i)] $\cup_{k=1}^\infty A_k=A$.
      \item[(ii)] For every $N\in[\nn]^\infty$ not containing two successive integers and
      every sequence $(x_k)_{k\in N}$ with $x_k\in A_k$ for all
     $n\in N$,  $(x_k)_{k\in N}$ is a Schauder basic sequence of
     constant $C$.
    \end{enumerate}
  \end{defn}
  The following proposition is well known but for the sake of
  completeness we outline its proof.
  \begin{prop}
    Let $(x_n)_{n}$ be a seminormalized weakly null sequence
    in a Banach space $X$. Then for every $\varepsilon>0$ the
    set $A=\{x_n: n\in\nn\}$ admits a SSD with constant
    $C=1+\varepsilon$.
  \end{prop}
  \begin{proof}
    We may assume that $X$ has a Schauder basis $(e_n)_{n}$ with
    basis constant K=1 (for example we may assume that $X=C[0,1]$).
    By induction and using the sliding hump argument, we define (1) a partition $(F_n)_{n}$ of $\nn$
    into finite pairwise disjoint sets and (2) a sequence
    $(y_n)_{n}$ of finitely supported vectors in $X$ such that
    the following are fulfilled:
    \begin{enumerate}
      \item[(i)]  For every $k\in\nn$ and $n\in F_k$, we have $\|x_n-y_n\|<\varepsilon/2^k$   and
      \item[(ii)] for every $k_2>k_1$ with  $k_2-k_1>1$, $n_1\in F_{k_1}$ and $n_2\in F_{k_2}$,
      we have
      $\max\text{supp}(y_{n_1})<\min\text{supp}(y_{n_2})$.
    \end{enumerate}
    Setting $A_k=\{x_n:n\in F_k\}$, $k\in\nn$, it is
    easy to check that $(A_k)_{k}$ satisfies conditions (i) and (ii) of Definition \ref{Def of SSD}.
  \end{proof}
   \begin{thm}\label{ssd giving Schauder basic}
    Let $A$ be a subset of a Banach space $X$. If $A$ admits a SSD with constant
    $C$, then every non trivial spreading model of any order of $A$
    is Schauder basic of constant $C$.
  \end{thm}
\begin{proof} Let $1\leq \xi$ be a countable ordinal and
$(e_n)_n$ be a non trivial spreading model of  $A$ of order $\xi$.
Let $\ff$ be a regular thin family with $o(\ff)=\xi$,
$(x_s)_{s\in\ff}$ be an $\ff$-sequence in $A$ and $M\in[\nn]^\infty$
such that $(x_s)_{s\in\ff\upharpoonright M}$ generates $(e_n)_{n}$
as an $\ff$-spreading model. Let $(A_k)_{k}$ be a partition of $A$
satisfying condition (ii) of Def. \ref{Def of SSD}. Finally, let
    $\varphi:\ff\upharpoonright M\to \nn$,
    defined  by $\varphi(s)=k$ if $x_s\in F_k$.

    Observe that $\varphi$ is hereditarily nonconstant in $M$. Indeed, otherwise
    there exists $L\in[M]^\infty$ and  $k_0\in\nn$ such that  $x_s\in F_{k_0}$ for every  $s\in\ff\upharpoonright L$
    and $s\in\ff\upharpoonright L$.
    Since $F_{k_0}$ is finite, by Proposition \ref{th2} there exists $N\in[L]^\infty$ such that
    $(x_s)_{s\in\ff\upharpoonright N}$ is constant. By part (ii) of Remark \ref{remark on the definition of spreading model},
     the $\ff$-sequence $(x_s)_{s\in\ff\upharpoonright N}$
    also generates $(e_n)_n$ as an $\ff$-spreading model. But then, since $(x_s)_{s\in\ff\upharpoonright N}$ is constant,
    the sequence $(e_n)_n$ should be trivial which is a contradiction.
    Hence $\varphi$ is hereditarily nonconstant in $M$ and therefore by Corollary \ref{Proposition combinatorial for SSD}
    there exists $N\in[M]^\infty$ such that $\varphi(s_2)-\varphi(s_1)>1$ for every plegma pair
    $(s_1,s_2)$ in $\ff\upharpoonright N$. By the SSD property of $A$ we have that for every $1\leq m <l\in\nn$ and for every
    \begin{equation}\|\sum_{j=1}^ma_jx_{s_j}\|\leq C\|\sum_{j=1}^la_jx_{s_j}\|,\end{equation}
    for every plegma $l$-tuple $(s_j)_{j=1}^l$ in $\ff\upharpoonright N$ and $a_1,...,a_l\in\rr$.
    This easily yields that $(e_n)_{n}$ is a Schauder basic sequence with constant $C$.
  \end{proof}

\subsection{Spreading models generated by subordinated
$\ff$-sequences} Let $(x_n)_n$ be a weakly convergent sequence  in a
Banach space $X$ which is not norm Cauchy and assume that it
generates a spreading model $(e_n)_n$. It is well known (see
\cite{BL}, \cite{BS}) that  $(e_n)_n$  is either an unconditional or
a singular spreading sequence. In \cite{AKT} we extended this fact
for subordinated $k$-sequences. Here we will show that similar
results also hold true for $\ff$-sequences where $\ff$ is a regular
thin family.
\subsubsection{Unconditional spreading models}
 Let $n\in\nn$ and for every $1\leq i\leq n$ let
$F_i\subseteq [\nn]^{<\infty}$. We will say that $(F_i)_{i=1}^n$ is
\textit{completely plegma connected} if for every choice of $s_i\in
F_i$, the $n$-tuple $(s_i)_{i=1}^n$ is  a plegma family. Also for a
subset $A$ of a Banach space $X$, $\text{conv} A$ denotes  the
convex hull of $A$.
\begin{lem}\label{Lemma finding convex  means}
  Let $X$ be a Banach space, $n\in\nn$, $\ff_1,\ldots,\ff_n$ be regular thin families,
 and $L\in[\nn]^\infty$.  Assume that   for every $i=1,..., n$,
 there exists  a continuous map  $\widehat{\varphi}_i:\widehat{\ff_i}\upharpoonright L\to (X,w)$.
 Then for every   $\varepsilon>0$ there exists a completely plegma connected family
  $(F_i)_{i=1}^n$  such that    $F_i\subseteq [\ff_i\upharpoonright L]^{<\infty}$
  and
  $\text{dist}\Big(\widehat{\varphi}_i(\varnothing), \text{conv
  }\widehat{\varphi}_i(F_i)\Big)<\varepsilon,$
  for every $i=1,...,n$.
\end{lem}
\begin{proof}
 We will use induction on $o\big((\ff_i)_{i=1}^l\big):=\max\{o(\ff_i):1\leq i\leq n\}$. If
  $o\big((\ff_i)_{i=1}^l\big)=0$, i.e. $\ff_i=\{\varnothing\}$ for all $1\leq i\leq n$ the result follows trivially.
   Let $1\leq \xi<\omega_1$
  and suppose that the lemma holds true if $o\big((\ff_i)_{i=1}^n\big)<\xi$.
  Let $n\in\nn$, $L\in[\nn]^\infty$ and
  $\ff_1,\ldots,\ff_n$ be regular thin families with
  $o\big((\ff_i)_{i=1}^l\big)=\xi$ and
  assume that  for every $1\leq i\leq n$, there exists a continuous map
 $\widehat{\varphi}_i:\widehat{\ff_i}\upharpoonright L\to (X,w)$.

 Fix $i\in\{1,...,n\}$. We may  suppose that $\ff_i $ is very large in $L$ and  therefore every singleton $\{l\}$ with $l\in L$
  belongs to $\widehat{\ff_i}$.  By the continuity of   $\widehat{\varphi}_i$,
  we get  that   $w\text{-}\lim_{l\in L}\widehat{\varphi}_i(\{l\})=\widehat{\varphi}_i(\varnothing)$. By Mazur's
  theorem,
   we may  choose a finite subset $\Lambda_i=\{l^i_1<...<l^i_{m_i}\}$  of $L$
    such that
\begin{equation}\label{bgg}
\text{dist}\Big(\widehat{\varphi}_i(\varnothing), \text{conv}\  \widehat{\varphi}_i(\Lambda_i)\Big)<\varepsilon/2,
\end{equation}
for every $1\leq i\leq n$. We may also assume that
\begin{equation}\label{bg}
\Lambda_1<...<\Lambda_n.
\end{equation}
  Let $\Lambda=\cup_{i=1}^n \Lambda_i$ and let $M=\{l\in L: l>\max\Lambda\}$.
  Fix $1\leq i\leq n$ and $1\leq j\leq m_i$. We set
   $\g^i_j=(\ff_i)_{(l^i_j)}=\{s\in [\nn]^{<\infty}: l^i_j< s \ \text{and} \ \{l^i_j\}\cup s\in\ff_i\}$ and
   let
  $\widehat{\varphi}^i_j:\widehat{\g^i_j}\upharpoonright M\to ( X, w)$, defined by
  $\widehat{\varphi}^i_j(s)=\widehat{\varphi}_{i}(\{l^i_j\}\cup s)$.
  Notice  that $\widehat{\varphi}^i_j$ is a continuous map and
since $o(\g^i_j)<o(\ff_i)$,
$o\Big(\big(\big(\g^i_j\big)_{j=1}^{m_i}\big)_{i=1}^n\Big)
  <o\big((\ff_i)_{i=1}^n\big)=\xi$. Therefore using our inductive assumption we may choose a
  completely plegma connected family
  $\big(\big(G^i_j\big)_{j=1}^{m_i}\big)_{i=1}^n$ such that
  $G^i_j\subseteq [ \g^i_j\upharpoonright M]^{<\infty}$
  and
  \begin{equation}\label{psiy}
\text{dist} \Big(\widehat{\varphi}^i_j(\varnothing), \text{conv}\ \widehat{\varphi}^i_j(G^i_j)\Big)<\varepsilon/2,
  \end{equation}
for every $1\leq i\leq  n$ and $1\leq j\leq m_i$.

  For every $1\leq i\leq n$ and $1\leq j\leq m_i$ we set
  $F^i_j=\big\{\{l^i_j\}\cup s: s\in G^i_j \big\}$.  By
equation (\ref{bg}) and the choice of
$\big(\big(G^i_j\big)_{j=1}^{m_i}\big)_{i=1}^n$, we easily see that
$\big(\big(F^i_j\big)_{j=1}^{m_i}\big)_{i=1}^n$ is  completely
plegma connected. Moreover,
 observe that  $\widehat{\varphi}^i_j(\varnothing)=\widehat{\varphi}_i(\{l^i_j\})$ and
  $\widehat{\varphi}^i_j(G^i_j)=\widehat{\varphi}_i(F^i_j)$. Hence,
equation (\ref{psiy}) translates to
 \begin{equation}\label{gti}
\text{dist} \Big(\widehat{\varphi}_i(\{l^i_j\}), \text{conv}\ \widehat{\varphi}_i(F^i_j)\Big)<\varepsilon/2,
  \end{equation}
for every $1\leq i\leq  n$ and $1\leq j\leq m_i$.

For every $1\leq i\leq n$, let $F_i=\bigcup_{j=1}^{m_i}F^i_j$.
Clearly $(F_i)_{i=1}^n$ is a completely  plegma connected family
with
 $F_i\subseteq [ \ff_i\upharpoonright L]^{<\infty}$, for every $i=1,...,n$.
Finally, fix $1\leq i\leq n$. By  (\ref{gti}) we have
$\text{dist}\big(x,\text{conv} \
\widehat{\varphi}_i(F_i)\big)<\varepsilon/2$, for every
$x\in\text{conv} \ \widehat{\varphi}_i(\Lambda_i)$ and therefore by
(\ref{bgg})  we conclude that $\text{dist}
\Big(\widehat{\varphi}_i(\varnothing), \text{conv}\
\widehat{\varphi}_i(F_i)\Big)<\varepsilon$.
\end{proof}

\begin{thm}\label{unconditional spreading model}
Let $X$ be a Banach space, $\ff$ be a regular thin family and
$L\in[\nn]^\infty$. Let $(x_s)_{s\in\ff\upharpoonright L}$ be an
$\ff$-subsequence in $X$ generating an $\ff$-spreading model
$(e_n)_n$. Also assume that  $(x_s)_{s\in\ff\upharpoonright L}$ is
seminormalized,  subordinated (with respect to the weak topology of
$X$) and  weakly null. Then $(e_n)_n$  is an $1$-unconditional
spreading sequence.
\end{thm}
\begin{proof} We first show that $(e_n)_{n}$ is non trivial. Indeed, otherwise by Theorem
\ref{Theorem equivalent forms for having norm on the spreading
  model} there exists $M\in[L]^\infty$ and $x_0\in X$ such that
  the $\ff$-subsequence $(x_s)_{s\in\ff\upharpoonright M}$
is norm   convergent to $x_0$. Since $M\subseteq L$,
$(x_s)_{s\in\ff\upharpoonright M}$ is also weakly null and therefore
$x_0=0$.  But this is a contradiction since
$(x_s)_{s\in\ff\upharpoonright L}$ is seminormalized.

We proceed to show that $(e_n)_n$ is $1$-unconditional. Fix
$n\in\nn$, $1\leq p\leq n$ and $a_1,\ldots,a_n\in [-1,1]$. It
suffices to show  that for every $\varepsilon >0$ we have
\begin{equation}\label{zs}\Big{\|}\sum_{\substack{i=1\\i\neq
p}}^na_ie_i\Big{\|}_*<\Big{\|}\sum_{i=1}^na_ie_i\Big{\|}_*+\varepsilon.\end{equation}
Indeed, fix $\varepsilon>0$. Since $(x_s)_{s\in\ff\upharpoonright
L}$ generates $(e_n)_n$ as an $\ff$-spreading model, by passing to a
final segment of $L$ if it is necessary we may assume that
\begin{equation}\label{eq8}
\Bigg{|} \Big{\|}\sum_{\substack{i=1\\i\neq p}}^na_ix_{s_i}\Big{\|} -\Big{\|}
\sum_{\substack{i=1\\i\neq p}}^na_ie_i\Big{\|}_* \Bigg{|}<\frac{\varepsilon}{3}
\;\;\text{and}\;\; \Bigg{|}\Big{\|}\sum_{i=1}^na_i
x_{s_i}\Big{\|}-\Big{\|}\sum_{i=1}^na_ie_i\Big{\|}_*
\Bigg{|}<\frac{\varepsilon}{3},
\end{equation}
 for every plegma $n$-tuple $(s_i)_{i=1}^n$ in $\ff\upharpoonright
L$.  Since $(x_s)_{s\in\ff\upharpoonright L}$ is subordinated with
respect to the weak topology, there exists a continuous map
$\widehat{\varphi}:\ff\upharpoonright L\to (X,w)$ such that
$\widehat{\varphi}(s)=x_s$ for every $s\in\ff\upharpoonright L$.
Since $(x_s)_{s\in\ff\upharpoonright L}$ is weakly convergent to
$\widehat{\varphi}(\varnothing)$ we have that
$\widehat{\varphi}(\varnothing)=0$. Therefore by Lemma \ref{Lemma
finding convex means} (for $\ff_i=\ff$ and
$\widehat{\varphi}_i=\widehat{\varphi}$, for all $i=1,...,n$),
 there exist a completely plegma connected family $(F_i)_{i=1}^n$
  and a sequence $(x_i)_{i=1}^n$ in $X$
  such that   $F_i\subseteq [\ff\upharpoonright L]^{<\infty}$, $x_i\in \text{conv}\
\widehat{\varphi}(F_i)$ and $\| x_i\|<\varepsilon/3$, for every
$1\leq i\leq n$. Let   $(\mu_s)_{s\in F_p}$ be a sequence in $[0,1]$
such that $\sum_{s\in F_p}\mu_s=1$ and $x_p=\sum_{s\in
F_p}\mu_s\widehat{\varphi}(t)$ and for each $i\neq p$ choose $s_i\in
F_i$. By the above we have   that for every $s\in F_p$ the $n$-tuple
$(s_1,\ldots,s_{p-1},s,s_{p+1},\ldots, s_n)$ is a plegma family and
$\|x_p\|=\|\sum_{s\in F_p}\mu_s x_s\|<\frac{\varepsilon}{3}$.

Therefore by  (\ref{eq8}) we have
  \[\begin{split}\Big{\|}\sum_{\substack{i=1\\i\neq p}}^na_ie_i\Big{\|}_* & \leq
  \Big{\|}\sum_{\substack{i=1\\i\neq p}}^na_ix_{s_i}\Big{\|}+\frac{\varepsilon}{3}
  \leq \Big{\|}\sum_{\substack{i=1\\i\neq p}}^na_ix_{s_i} +a_p\sum_{s\in F_p}\mu_s x_s\Big{\|}+
  |a_p|\frac{\varepsilon}{3}+\frac{\varepsilon}{3}\\
  &\leq\sum_{s\in F_p}\mu_s \Big{\|}\sum_{\substack{i=1\\i\neq p}}^na_ix_{s_i} +a_p x_s\Big{\|}+\frac{2\varepsilon}{3}
  \leq \sum_{s\in F_p}\mu_s \Big{(}\Big{\|}\sum_{i=1}^n a_ie_i \Big{\|}_* +\frac{\varepsilon}{3} \Big{)}+\frac{2\varepsilon}{3} \\
  &=\Big{\|}\sum_{i=1}^n a_ie_i \Big{\|}_* +\varepsilon.
    \end{split} \] The proof is complete.
\end{proof}

\subsubsection{Singular or isomorphic to $\ell_1$ spreading models}
We proceed to show an analogue of Theorem \ref{unconditional
spreading model} for subordinated $\ff$-sequences which are not
weakly null. We will need the following lemma.
\begin{lem}\label{isometric 1-subsym sequences}
  Let $(e_n)_{n}$ and $(\widetilde{e}_n)_{n}$ be two
  non trivial spreading sequences which are
  both  Ces\'aro summable to zero. Suppose that for
  every $n\in\nn$ and $\lambda_1,\ldots,\lambda_n$ with
  $\sum_{i=1}^n\lambda_i=0$, we have that
  \begin{equation}\label{pu}
  \Big\|\sum_{i=1}^n\lambda_ie_i\Big\|_*=\Big\|\sum_{i=1}^n\lambda_i\widetilde{e}_i\Big\|_{**}.\end{equation}
  Then  the map $e_n\to\widetilde{e}_n$ extends to a linear isometry from $<(e_n)_{n}>$ onto  $<(\widetilde{e}_n)_{n}>$.
\end{lem}
\begin{proof}
  Let $n\in\nn$ and $\lambda_1,\ldots,\lambda_n\in \rr$. Since $(e_n)_{n}$ (resp.
  $(\widetilde{e}_n)_{n}$) is Ces\'aro summable to zero, we
  have
$\lim_{m\to\infty}\frac{1}{m}\sum_{j=1}^me_{n+j}=0$ (resp.
$\lim_{m\to \infty}\frac{1}{m}\sum_{j=1}^m\widetilde{e}_{n+j}=0$).
Let $\lambda=\sum_{i=1}^n\lambda_i$. Then
$\sum_{i=1}^n\lambda_i-\sum_{j=1}^m\frac{\lambda}{m}=0$ and
therefore,
\[\begin{split}\Big\|\sum_{i=1}^n\lambda_ie_i\Big\|_*&=
\lim_{m\to\infty}\Big\|\sum_{i=1}^n\lambda_ie_i-\frac{\lambda}{m}\sum_{j=1}^me_{n+j}\Big\|_*\\
&\stackrel{(\ref{pu})}{=}\lim_{m\to\infty}\Big\|\sum_{i=1}^n\lambda_i\widetilde{e}_i-\frac{\lambda}{m}\sum_{j=1}^m\widetilde{e}_{n+j}\Big\|_{**}
=\Big\|\sum_{i=1}^n\lambda_i\widetilde{e}_i\Big\|_{**}\end{split}\]
\end{proof}

The next lemma is from \cite{AKT}. We reproduce it for the sake of
completeness.
\begin{lem}\label{triv-ell}
Let $X$ be  a Banach space, $\ff$ be a regular thin family and
$(x_s)_{s\in\ff\upharpoonright L}$ be an $\ff$-subsequence in $X$.
Let $x_0\in X$ and set $x'_s=x_s-x_0$, for all
$s\in\ff\upharpoonright L$. Assume that
$(x_s)_{s\in\ff\upharpoonright L}$ and
$(x'_s)_{s\in\ff\upharpoonright L}$ generate $\ff$-spreading models
$(e_n)_n$ and $(\widetilde{e}_n)_n$ respectively. Then the following
hold.
\begin{enumerate}
\item[(i)] $\|\sum_{i=1}^na_ie_{_i}\|=\|\sum_{i=1}^na_i\widetilde{e}_{_i}\|$,  for  every $n\in\nn$ and  $a_1,\ldots,a_n\in\rr$ with
$\sum_{i=1}^na_i=0$.
\item[(ii)] The sequence $(e_n)_n$ is trivial if and only if $(\widetilde{e}_n)_n$ is trivial.
\item[(iii)] The sequence $(e_n)_n$ is equivalent to the usual basis of $\ell^1$
if and only if $(\widetilde{e}_n)_n$ is equivalent to the usual
basis of $\ell^1$.
\end{enumerate}
\end{lem}
\begin{proof}
(i) Notice  that for  every $n\in\nn$, $s_1,...,s_n$ in
$\ff\upharpoonright L$ and $a_1,\ldots,a_n\in\rr$ with
$\sum_{i=1}^na_i=0$, we have
$\sum_{i=1}^na_ix_{s_i}=\sum_{i=1}^na_ix'_{s_i}$. Since $(e_n)_n$
and  $(\widetilde{e}_n)_n$ are generated by $(x_s)_{s\in
\ff\upharpoonright L}$
and $(x_s)_{s\in \ff\upharpoonright L}$ the result follows. \\
(ii) It  follows by part (i) and Lemma \ref{sing}.\\
(iii) We fix $\ee>0$. If $(\widetilde{e}_n)_n$ is not equivalent to
the usual basis of $\ell^1$ then there exist $n\in\nn$ and
$a'_1,\ldots,a'_n\in\rr$ such that $\sum_{i=1}^n|a'_i|=1$ and
$\|\sum_{i=1}^na'_i\widetilde{e}_i\|<\ee$. Setting $a_i=a'_i/2$ and
$a_{n+i}=-a'_i/2$, for all $1\leq i\leq n$, we have
$\sum_{i=1}^{2n}a_i=0$ and therefore,
$\|\sum_{i=1}^{2n}a_ie_i\|=\|\sum_{i=1}^{2n}a_i\widetilde{e}_i\|<\ee$.
Since  $\sum_{i=1}^{2n}|a_i|=1$,  $(e_n)_n$ is also not equivalent
to the usual basis of $\ell^1$.
\end{proof}
\begin{thm}\label{lemma either l1 or not Schauder basic}
Let $X$ be a Banach space, $\ff$ be a regular thin family and
$L\in[\nn]^\infty$. Let $(x_s)_{s\in\ff\upharpoonright L}$ be an
$\ff$-subsequence in $X$ generating a non trivial $\ff$-spreading
model $(e_n)_n$. Also assume that  $(x_s)_{s\in\ff\upharpoonright
L}$ is subordinated and let $x_0$ be the weak-limit of
$(x_s)_{s\in\ff\upharpoonright L}$. Finally, let $x'_s=x_s-x_0$, for
every $s\in\ff\upharpoonright L$. If   $x_0\neq 0$ then exactly one
of the following holds.
\begin{enumerate}
\item[(i)] The sequence $(e_n)_n$ as well as  every spreading model of $(x'_s)_{s\in\ff\upharpoonright L}$
 is  equivalent to the usual basis of $\ell^1$. \item[(ii)] The
sequence $(e_n)_n$ is singular and if $e_n=e'_n+e$ is its
natural decomposition  then $\|e\|=\|x_0\|$ and $(e'_n)_n$ is
the unique (up to isometry) $\ff$-spreading model of
$(x'_s)_{s\in \ff\upharpoonright L}$.
\end{enumerate}
\end{thm}
\begin{proof}
Let  $(\widetilde{e}_n)_n$ be an $\ff$-spreading model of
$(x'_s)_{s\in\ff\upharpoonright L}$. By passing to an infinite
subset of $L$ if it is necessary we may assume that
$(x'_s)_{s\in\ff\upharpoonright L}$ generates $(\widetilde{e}_n)_n$
as an $\ff$-spreading model.

If $(e_n)_n$ is equivalent to the usual basis of $\ell^1$ then by
Lemma \ref{triv-ell}, we have that the same holds for
$(\widetilde{e}_n)_n$ and hence  (i) is satisfied. Otherwise, again
by Lemma \ref{triv-ell}, $(\widetilde{e}_n)_n$ is also non trivial
and not equivalent to the $\ell^1$-basis.  Let us denote by
$\|\cdot\|_*$ (resp. $\|\cdot\|_{**}$) the norm of the space
generated by $(e_n)_n$ (resp. $(\widetilde{e}_n)_n$). Since
$(\widetilde{e}_n)_n$ is non trivial, we have that
$\|\widetilde{e}_n\|_{**}>0$ and therefore (by passing to a final
segment of $L$ if it is necessary) we may assume that
$(x'_s)_{s\in\ff\upharpoonright L}$ is seminormalized. It is also
easy to see that $(x'_s)_{s\in\ff\upharpoonright M}$ is subordinated
and weakly null. Therefore by Theorem \ref{unconditional spreading
model}, $(\widetilde{e}_n)_n$ is 1-unconditional. Moreover, since
$(\widetilde{e}_n)_n$ is not equivalent to the usual basis of
$\ell^1$, by  Proposition \ref{equiv forms for 1-subsymmetric weakly
null} (ii), we conclude that $(\widetilde{e}_n)_n$ is norm Ces\`aro
summable to zero. Hence, by part (i) of Lemma \ref{triv-ell}, we
have
\begin{equation}\label{tv}
\lim_{n\to\infty}\Big\|\frac{1}{n}\sum_{j=1}^ne_j-\frac{1}{n}\sum_{j=n+1}^{2n}e_j\Big\|_*
=\lim_{n\to\infty}\Big\|\frac{1}{n}\sum_{j=1}^n\widetilde{e}_j-\frac{1}{n}\sum_{j=n+1}^{2n}\widetilde{e}_j\Big\|_{**}=0.
\end{equation}
For every $n\in\nn$ choose a sequence
$(s_j^n)_{j=1}^n\in\text{Plm}_n(\ff\upharpoonright L)$ such that $
\min s_1^n\geq L(n)$.  Since $x_s-x_0=x'_s$, for every $s\in
\ff\upharpoonright L$, we have
\begin{equation}
\lim_{n\to\infty}\Big\|x_0-\frac{1}{n}\sum_{j=1}^nx_{s_j^n}\Big\|=\lim_{n\to\infty}\Big\|\frac{1}{n}\sum_{j=1}^nx'_{s_j^n}\Big\|=
\lim_{n\to\infty}\Big\|\frac{1}{n}\sum_{j=1}^n\widetilde{e}_n\Big\|_{**}=0.
\end{equation}
Therefore
\begin{equation}\label{fd}
\lim_{n\to\infty}\Big\|\frac{1}{n}\sum_{j=1}^ne_j\Big\|_*=\lim_{n\to\infty}\Big\|\frac{1}{n}\sum_{j=1}^nx_{s_j^n}\Big\|
=\|x_0\|>0.
\end{equation}
By (\ref{tv}) and (\ref{fd}), we get that $(e_n)_n$ is not Schauder
basic, i.e. it is singular. Let $e_n=e'_n+e$ be the natural
decomposition of $(e_n)_n$. By (\ref{fd}) and the fact that
$(e'_n)_n$ is Ces\`aro summable to zero, we have that
$\|e\|=\|x_0\|$. To complete the proof it remains to show that
$(\widetilde{e}_n)_n$ and $(e'_n)_n$ are isometrically equivalent.
By Lemma \ref{isometric 1-subsym sequences} it suffices to show that
\begin{equation}\Big\|\sum_{i=1}^n\lambda_ie_i'\Big\|_*=\Big\|\sum_{i=1}^n\lambda_i\widetilde{e}_i\Big\|_{**},\end{equation}
for every $n\in\nn$ and $\lambda_1,\ldots\lambda_n\in\rr$ with
$\sum_{i=1}^n\lambda_1=0$. Indeed, fix $n\in\nn$ and
$\lambda_1,\ldots\lambda_n\in\rr$ with $\sum_{i=1}^n\lambda_i=0$.
For each $k\in\nn$ choose
$(s_j^k)_{j=1}^n\in\text{{Plm}}_n(\ff\upharpoonright L)$ such that
$\lim_{k\to\infty}\min s_1^k=+ \infty$. Then
\[\Big\|\sum_{i=1}^n\lambda_ie_i'\Big\|_*=\Big\|\sum_{i=1}^n\lambda_ie_i\Big\|_*=\lim_{k\to\infty}\Big\|\sum_{i=1}^n\lambda_ix_{s^k_i}\Big\|=
\lim_{k\to\infty}\Big\|\sum_{i=1}^n\lambda_ix'_{s^k_i}\Big\|=\Big\|\sum_{i=1}^n\lambda_i\widetilde{e}_i\Big\|_{**}\]
 and the proof is complete.
\end{proof}
\subsubsection{Weakly relatively compact $\ff$-sequences }
 Let $X$ be a Banach space and $\xi<\omega_1$. By
  $\mathcal{SM}_\xi^{wrc}(X)$ we will denote
  the set of all spreading sequences
    $(e_n)_{n}$ such that there exists
     a weakly relatively compact subset $W$ of
    $X$ which  admits $(e_n)_{n}$ as a $\xi$-spreading model.
    We  also set\[\mathcal{SM}^{wrc}(X)=\bigcup_{\xi<\omega_1}\mathcal{SM}_\xi^{wrc}(X).\]
Hence  $(e_n)_n\in \mathcal{SM}_\xi^{wrc}(X)$ if and only if there
exists an  $\ff$-sequence $(x_s)_{s\in\ff}$ such that
$\overline{\{x_s: s\in\ff\}}^{w}$ is a weakly compact subset of $X$
and for some $L\in [\nn]^\infty$, $(x_s)_{s\in\ff\upharpoonright L}$
generates $(e_n)_n$ as an $\ff$-spreading model. The $\ff$-sequences
with weakly relatively compact range will be called \emph{weakly
relatively compact} (``wrc'' in short).  The following proposition
says that every wrc $\ff$-sequence always contains  a subordinated
subsequence.
\begin{prop}\label{cor for subordinating}
  Let $X$ be a Banach space, $\ff$ be a regular thin family and
  $(x_s)_{s\in\ff}$ be a weakly relatively compact $\ff$-sequence in
  $X$. Then for every $M\in[\nn]^\infty$ there exists
  $L\in[M]^\infty$ such that the $\ff$-subsequence $(x_s)_{s\in\ff\upharpoonright L}$ is
  subordinated with respect to the weak topology.
  % Consequently, $1\leq \zeta<\xi$ be countable ordinals. Then $\mathcal{SM}_\zeta^{wrc}(X)\subseteq \mathcal{SM}_\xi^{wrc}(X)$.
 \end{prop}
 \begin{proof}
 Let $M\in[\nn]^\infty$.   Since the
weak topology on every separable weakly  compact subset of a Banach
space is metrizable, we have that $\overline{\{x_s: s\in\ff\}}^{w}$
is compact metrizable. By Theorem \ref{Create subordinated} the
result follows.
 \end{proof}

\begin{prop}\label{wrc}
 Let $X$ be a Banach space, $\xi<\omega_1$ and $(e_n)_{n}\in
\mathcal{SM}_\xi^{wrc}(X)$. Then for every regular thin family $\g$
with $o(\g)\geq \xi$ there exist
  a weakly relatively compact $\g$-sequence $(w_t)_{t\in\g}$ in $X$ and   $L\in[\nn]^\infty$
such that $(w_t)_{t\in\g\upharpoonright L}$ is
  subordinated with respect to the weak topology and generates
  $(e_n)_{n}$ as a $\g$-spreading model. Consequently  $\mathcal{SM}_\zeta^{wrc}(X)\subseteq
\mathcal{SM}_\xi^{wrc}(X)$,  for every $1\leq \zeta<\xi<\omega_1$.
\end{prop}
\begin{proof}
   Since $(e_n)_{n}\in \mathcal{SM}_\xi^{wrc}(X)$ there
   exists a weakly relatively compact subset $A$ of
    $X$ such that $A$ admits $(e_n)_{n}$ as a $\xi$-spreading model. Hence there exists a regular thin family
    $\ff$ of order $\xi$, an $\ff$-sequence $(x_s)_{s\in\ff}$ in $A$ and $M\in[\nn]^\infty$
   such that $(x_s)_{s\in\ff\upharpoonright M}$ generates
  $(e_n)_{n}$ as an $\ff$-spreading model. By  Lemma
  \ref{propxi} there exist  a $\g$-sequence
  $(w_t)_{t\in\g}$ and $N\in[\nn]^\infty$  such that $(w_t)_{t\in\g\upharpoonright N}$
  generates $(e_n)_{n}$ as a $\g$-spreading model and moreover
  $\{w_t:t\in\g\}\subseteq\{x_s:s\in\ff\}\subseteq A$.
  Hence  $(w_t)_{t\in\g}$
  is a weakly relatively compact $\g$-sequence. By Proposition \ref{cor for subordinating}
  there exists $L\in[N]^\infty$ such that  $(w_t)_{t\in\g\upharpoonright L}$  is
  subordinated with respect to the weak topology. Clearly   $(w_t)_{t\in\g\upharpoonright L}$ also generates
  $(e_n)_{n}$ as a $\g$-spreading model and the proof is
  complete.
\end{proof}
Proposition \ref{wrc} implies that every  $(e_n)_{n}$ in $
\mathcal{SM}^{wrc}(X)$ is generated by a subordinated
$\ff$-subsequence. Hence, by  Theorems \ref{unconditional spreading
model} and  \ref{lemma either l1 or not Schauder basic} we obtain
the following.
\begin{cor}\label{relatively weakly compact sets have
unconditional spreading models}
  Let $X$ be a Banach space, $\ff$ be a regular thin family and $(x_s)_{s\in\ff}$ be a weakly
  relatively compact $\ff$-sequence. Let $(e_n)_n$ be a spreading sequence
  and assume that $(x_s)_{s\in\ff}$ admits  $(e_n)_{n}$ as an $\ff$-spreading model. Then exactly
  one of the following holds:
  \begin{enumerate}
    \item[(i)] The sequence $(e_n)_{n}$ is trivial.
    \item[(ii)] The sequence $(e_n)_{n}$ is singular. In this case there exist
     $L\in[\nn]^\infty$ and $x_0\in X$ such that if $e_n=e'_n+e$
     is the natural decomposition of $(e_n)_{n}$ then the
    $\ff$-subsequence $(x'_s)_{s\in \ff\upharpoonright L}$,
    defined by $x'_s=x_s-x_0$ for all $s\in\ff\upharpoonright
    L$, generates the sequence $(e'_n)_{n}$ as an
    $\ff$-spreading model and $\|x_0\|=\|e\|$.
    \item[(iii)] The sequence $(e_n)_{n}$ is Schauder basic. In this case $(e_n)_{n}$ is unconditional.
  \end{enumerate}
\end{cor}

\subsection{$\ff$-sequences generating singular spreading models}
Let $X$ be a Banach space and $(x_n)_{n}$ be a
  sequence in $X$ which generates a singular spreading model
  $(e_n)_{n}$ and let $e_n=e'_n+e$ be the natural decomposition of
  $(e_n)_{n}$. It can be shown that
    there exists $x\in X\setminus \{0\}$
  such that  $\|x\|=\|e\|$ and setting
  $x'_n=x_n-x$, $(e'_n)_{n}$ is the unique spreading model of
  $(x'_n)_{n}$. In the following we will present an extension of this fact for
  $\ff$-sequences in a Banach space $X$. We start with the
  next lemma.

\begin{lem} \label{cvh}  Let $\ff$ be a regular thin family, $M\in[\nn]^\infty$ and
  $(x_s)_{s\in\ff}$ be an $\ff$-sequence in a Banach space $X$
  such that $(x_s)_{s\in\ff\upharpoonright M}$ generates a
  singular $\ff$-spreading model $(e_n)_n$. Then there exists $L\in[M]^\infty$ satisfying the
  next
property.

For every $\ee>0$ there exists $m_0\in\nn$ such that
 \begin{equation}
 \label{eqstar}\Big\|\frac{1}{n}\sum_{j=1}^nx_{s_j}-\frac{1}{m}\sum_{j=1}^mx_{t_j}\Big\|<\ee,
 \end{equation}
for every $n,m\geq m_0$ and
$(s_j)_{j=1}^n,(t_j)_{j=1}^m\in\text{Plm}(\ff\upharpoonright
  L)$ with $s_1(1)\geq L(n)$ and $t_1(1)\geq L(m)$.
\end{lem}
  \begin{proof} Initially we notice that a weaker version of the lemma holds true, that
is for every $\ee>0$ there exists $k_0\in\nn$ such
  that for every $n,m\geq k_0$ and every
  $(s_j)_{j=1}^{n+m}\in\text{{Plm}}(\ff\upharpoonright M)$
  with $s_1(1)\geq M(n+m)$, we have
 \begin{equation}\label{pk}\Big\|\frac{1}{n}\sum_{j=1}^nx_{s_j}-\frac{1}{m}\sum_{j=1}^mx_{s_{n+j}}\Big\|<\ee.\end{equation}
Indeed, let $\ee>0$. Since $(e_n)_n$ is singular, it is weakly
convergent to some $e$
 and moreover setting $e'_n=e_n-e$, the sequence   $(e'_n)_{n}$ is
 Ces\'aro summable to zero. Hence we may  choose $n_0\in\nn$
such that  $\Big\|\frac{1}{n}\sum_{i=1}^n e'_i\Big\|_*<\ee/4$, for
all $n\geq n_0$. Hence for every $n,m\geq n_0$, we have
\begin{equation}\Big\|\frac{1}{n}\sum_{i=1}^n e_i-\frac{1}{m}\sum_{i=1}^m
e_{n+i}\Big\|_*= \Big\|\frac{1}{n}\sum_{i=1}^n
e'_i-\frac{1}{m}\sum_{i=1}^m e'_{n+i}\Big\|_*<\ee/2\end{equation}
Since $(x_s)_{s\in\ff\upharpoonright M}$ generates $(e_n)_n$ as an
$\ff$-spreading model we can find $k_0\geq n_0$ such that for
every $n,m\geq k_0$ and every
  $(s_j)_{j=1}^{n+m}\in\text{{Plm}}(\ff\upharpoonright M)$
  with $s_1(1)\geq M(n+m)$ equation
(\ref{pk}) is satisfied.

Let $(\ee_k)_k$ be a sequence of positive real numbers such that
$\sum_k\ee_k<+\infty$.  By the above we can choose an increasing
sequence $(n_k)_{k}$ in $\nn$ such that for every $k\in\nn$,
$n,m\geq n_k$ and
$(s_j)_{j=1}^{n+m}\in\text{Plm}(\ff\upharpoonright M)$ with
$s_1(1)\geq M(n+m)$, it holds  that
\begin{equation}\label{df}\Big\|\frac{1}{n}\sum_{j=1}^nx_{s_j}-\frac{1}{m}\sum_{j=1}^mx_{s_{n+j}}\Big\|< \ee_k.\end{equation}
 We may also assume that $\ff$ is very large in
$M$ and $2n_k<n_{k+1}$, for every $k\in\nn$.

We set  $L=\{M(2n_k+n_{k+1}):k\in\nn\}$ and we shall show that $L$
satisfies the conclusion of the lemma. To this end we shall use an
appropriate map sending each $s\in\ff\upharpoonright L$ to a
plegma family in $\ff\upharpoonright M$. First, for every
$s\in\ff\upharpoonright L$ and $p=1,...,|s|$, let $k(s(p))$ be the
unique positive integer $k$ satisfying
$s(p)=L(k)=M(2n_{k}+n_{k+1})$.  We define $\Phi:
\ff\upharpoonright L\to \text{Plm}(\ff\upharpoonright M)$ as
follows. For every $s\in\ff\upharpoonright L$, we assign the
$n_{k(s(1))}$-tuple $\Phi(s)=\big(v^s_j\big)_{j=1}^{n_{k(s(1))}}$
where  $v^s_j$ is the unique element of $\ff\upharpoonright M$,
satisfying
 \begin{equation}
   v^s_j\sqsubseteq
   \big\{M\big(2n_{k(s(p))}+n_{k(s(p))+1}-n_{k(s(1))}+j\big):p=1,\ldots,
   |s|\big\}.
 \end{equation}
The existence of  $v^s_j$,   $j=1,...,
n_{k(s(1))}$, follows easily from the fact  that $\ff$ is  regular
thin and very large in $M$.

Below we state some useful properties of $\Phi$. Their
verification   is straightforward.
\begin{enumerate}
\item[(P1)] For every $s\in\ff\upharpoonright L$, $\Phi(s)\in
\text{Plm}_{n_k}(\ff\upharpoonright M)$,
$v_1^s(1)>M(n_{k}+n_{k+1})$ and $v^s_{n_{k}}= s$, where
$k=k(s(1))$, \item[(P2)] for every $(s_1,s_2)\in
\text{Plm}_2(\ff\upharpoonright L)$, the concatenation
$\Phi(s_1)^\frown\Phi(s_2)$  belongs to
$\text{Plm}(\ff\upharpoonright M)$.
\end{enumerate}
We are now ready to prove that $L$ is actually the desired set.
Fix  a positive integer $k$ and let us denote by   $s$  the unique
element of $\ff\upharpoonright L$ such that $s\sqsubseteq
\{L(i):i\geq k\}$. Notice that $s(1)=L(k)=M(2n_k+n_{k+1})$ and
therefore $k(s(1))=k$. Also let
 $m_k=\max\{n_k,
k+|s|+1\}$.  We claim that
\begin{equation}\label{df3}\Big\|\frac{1}{n_{k}}
\sum_{j=1}^{n_{k}}x_{v^{s}_j}-\frac{1}{m}\sum_{j=1}^mx_{t_j}\Big\|<\sum_{l=k}^{k+|s|}
\ee_l,\end{equation} for every $m\geq m_k$ and
$(t_j)_{j=1}^m\in\ff\upharpoonright L$ with $t_1(1)\geq L(m)$.

Indeed, let  $m\geq m_k$ and $(t_j)_{j=1}^m\in\ff\upharpoonright
L$ with $t_1(1)\geq L(m)$. Notice that $\max s=L(k+|s|-1)<L(m)\leq
t_1(1)=\min t_1$. Hence, by  Theorem \ref{accessing everything
with plegma path of length |s_0|} there exists a plegma path
$(w_l)_{l=0}^{l_0}$ in $\ff\upharpoonright L$ from $w_0=s$ to
$w_{l_0}=t_1$ of length $l_0=|s|$.  Notice that  $k(w_l(1))\geq
k+l$ which implies that $n_{k(w_l(1))}\geq n_{k+l}$ and therefore
$\big(v^{w_l}_1,\ldots,v^{w_l}_{n_{k+l}}\big)$ is a subfamily of
$\Phi(w_l)$. Thus, by properties (P1) and (P2) above, we have that
$\big(v^{w_l}_1,\ldots,v^{w_l}_{n_{k+l}},v^{w_{l+1}}_1,\ldots,v^{w_{l+1}}_{n_{k+l+1}}\big)$
is a plegma family in $\ff\upharpoonright L$ of length
$n_{k+l}+n_{k+l+1}$ with $v^{w_l}_1(1)>M(n_
{k(w_l(1))}+n_{k(w_l(1)+1)})>M(n_{k+l}+n_{k+l+1})$. Hence by
(\ref{df}) we get
\begin{equation}\Big\|\frac{1}{n_{k+l}}\sum_{j=1}^{n_{k+l}}x_{v^{w_l}_j}-\frac{1}{n_{k+l+1}}
\sum_{j=1}^{n_{k+l+1}}x_{v^{w_{l+1}}_j}\Big\|<\ee_{k+l},
\end{equation} for every $l=0,...,l_0-1$.
 Thus,
\begin{equation}\label{df1}\Big\|\frac{1}{n_{k}}\sum_{j=1}^{n_k}x_{v^{s}_j}-\frac{1}{n_{k+|s|}}
\sum_{j=1}^{n_{k+|s|}}x_{v^{t_1}_j}\Big\|<\sum_{l=k}^{k+|s|-1}\ee_{l}.\end{equation}
Similarly, since $m>k+|s|=k+|l_0|$ we have that $n_m>n_{k+|s|}$.
Also since $t_1(1)\geq L(m)=M(2n_m+ n_{m+1})$ we have that
$k(t_1(1))\geq m $.  Hence $n_{k(t_1(1))}\geq n_m>n_{k+|s|}$ which
implies that $\big(v^{t_1}_1,...,v^{t_1}_{n_{k+|s|}}\big)$ is a
proper subfamily of $\Phi(t_1)$. Therefore,
$\big(v^{t_1}_1,...,v^{t_1}_{n_{k+|s|}}, t_1,...,t_m\big)$ is a
plegma family in $\ff\upharpoonright L$. Moreover  $t_1(1)\geq
M(2n_m+ n_{m+1})\geq M(n_{k+|s|}+m)$ and so again by (\ref{df}),
we have
\begin{equation}\label{df2} \Big\|\frac{1}{n_{k+|s|}}\sum_{j=1}^{n_{k+|s|}}x_{v^{t_1}_j}-
\frac{1}{m}\sum_{j=1}^mx_{t_j}\Big\|<\ee_{k+|s|}.\end{equation}
Now (\ref{df3}) follows by (\ref{df1}) and (\ref{df2}).

Finally, by (\ref{df3}) and a triangle inequality  we obtain that
 \begin{equation}
 \Big\|\frac{1}{n}\sum_{j=1}^nx_{s_j}-\frac{1}{m}\sum_{j=1}^mx_{t_j}\Big\|<2\sum_{l=k}^{k+|s|}\ee_l,
 \end{equation}
for every $k\in\nn$, $n,m\geq m_k$ and
$(s_j)_{j=1}^n,(t_j)_{j=1}^m\in\text{Plm}(\ff\upharpoonright
  L)$ with $s_1(1)\geq L(n)$ and $t_1(1)\geq L(m)$. Since
  $\sum_k\ee_k<+\infty$ the proof is complete.
\end{proof}

\begin{thm}\label{singhighord}
  Let $\ff$ be a regular thin family, $M\in[\nn]^\infty$ and
  $(x_s)_{s\in\ff}$ be an $\ff$-sequence in a Banach space $X$
  such that $(x_s)_{s\in\ff\upharpoonright M}$ generates a
  singular $\ff$-spreading model $(e_n)_{n}$. Let $e_n=e'_n+e$ be the natural decomposition
  of $(e_n)_{n}$. Then there
  exist $x\in X$  with $\|x\|=\|e\|_*$ and $L\in[M]^\infty$ such that setting
  $x'_s=x_s-x$ the
  $\ff$-subsequence
  $(x'_s)_{s\in\ff\upharpoonright L}$ admits $(e'_n)_{n}$ as
  a unique (up to isometry) $\ff$-spreading model.
\end{thm}
\begin{proof} We start by determining the element  $x\in X$. Let $L\in[M]^\infty$ satisfying Lemma \ref{cvh}. For every $k\in\nn$
we set\[A_k=\Big\{\frac{1}{n}\sum_{i=1}^n x_{s_i}:
\;(s_i)_{i=1}^n\in\text{{Plm}}(\ff\upharpoonright
L)\;\text{and}\;s_1(1)\geq n\geq k \Big\}.\] Clearly the sequence
$(A_k)_{k}$ is decreasing and  by Lemma \ref{cvh},
$\text{diam}(A_k)\to 0$. Therefore there exists a unique $x\in X$
such that $\cap_{k=1}^\infty \overline{A_k}=\{x\}$.

We continue to   show that $\|e\|=\|x\|$. Notice that by the
choice of $x$, we have that for every $\ee>0$ there exists
$n_0\in\nn$ such that for all $n\geq n_0$
  \begin{equation}\label{eq2star}\Big\|\frac{1}{n}\sum_{j=1}^nx_{s_j}-x\Big\|<\ee.\end{equation}
For each $n\in\nn$ we pick
$(s_i^n)_{i=1}^n\in\text{{Plm}}(\ff\upharpoonright L)$, with
$s_i^n(1)\geq L(n)$. By (\ref{eq2star}), we have
\begin{equation}\label{xd}\lim_n\Big\|\frac{1}{n}\sum_{i=1}^n
x_{s_i^n}-x\Big\|=0.\end{equation} Also, since
$(x_s)_{s\in\ff\upharpoonright L}$   generates $(e_n)_{n}$ as an
$\ff$-spreading model, we get
\begin{equation}\label{nk}\lim_n\Bigg|\Big\|\frac{1}{n}\sum_{i=1}^n
x_{s_i^n}\Big\|-\Big\|\frac{1}{n}\sum_{i=1}^n
e_i\Big\|_*\Bigg|=0.\end{equation}
 Moreover, since
$(e'_n)_{n}$ is Ces\`aro summable to zero, we have
\begin{equation}\label{rd}\lim_n\Big\|\frac{\sum_{i=1}^n e_i}{n}-e\Big\|_*=0.\end{equation}
Hence,
\[\|e\|_*\stackrel{(\ref{rd})}{=}\lim_n\Big\|\frac{1}{n}\sum_{i=1}^n e_i\Big\|_*\stackrel{(\ref{nk})}{=}
\lim_n\Big\|\frac{1}{n}\sum_{i=1}^n
x_{s_i^n}\Big\|\stackrel{(\ref{xd})}{=}\|x\|.\]

 We proceed now to show  that
$(e'_n)_{n}$ is the unique $\ff$-spreading model of
$(x'_s)_{s\in\ff\upharpoonright L}$, where $x'_s=x_s-x$,
$s\in\ff\upharpoonright L$. Let $N\in [L]^\infty$  such that
$(x'_s)_{s\in\ff\upharpoonright N}$ generates an $\ff$-spreading
model $(\widetilde{e}_n)_{n}$. We will show that
$(\widetilde{e}_n)_{n}$ is isometric to $(e'_n)_{n}$. Since,
 \begin{equation}\frac{1}{n}\sum_{j=1}^nx'_{s_j}=\frac{1}{n}\sum_{j=1}^nx_{s_j}-x,\end{equation}
for every $n\in\nn$, by (\ref{eq2star}) we conclude that
$(\widetilde{e}_n)_{n}$ is Ces\'aro summable to zero. Hence  by
Lemma \ref{isometric 1-subsym sequences} it suffices to show that
\begin{equation}\Big\|\sum_{i=1}^n\lambda_ie_i'\Big\|_*=\Big\|\sum_{i=1}^n\lambda_i\widetilde{e}_i\Big\|_{**},\end{equation}
for every $n\in\nn$ and $\lambda_1,\ldots\lambda_n\in\rr$ with
$\sum_{i=1}^n\lambda_1=0$. Indeed, let $n\in\nn$ and
$\lambda_1,\ldots\lambda_n\in\rr$ with $\sum_{i=1}^n\lambda_i=0$.
Also let $((s_j^k)_{j=1}^n)_{k}$ be a sequence in
$\text{{Plm}}_n(\ff\upharpoonright L)$ such that $
\lim_{k\to\infty}s_1^k(1)=+ \infty$. Then
\[\Big\|\sum_{i=1}^n\lambda_ie_i'\Big\|_*=\Big\|\sum_{i=1}^n\lambda_ie_i\Big\|_*=\lim_{k\to\infty}\Big\|\sum_{i=1}^n\lambda_ix_{s^k_i}\Big\|=
\lim_{k\to\infty}\Big\|\sum_{i=1}^n\lambda_ix'_{s^k_i}\Big\|=\Big\|\sum_{i=1}^n\lambda_i\widetilde{e}_i\Big\|_{**}\]
and the proof is complete.
\end{proof}
 We close  by  a strengthening of Theorem \ref{singhighord} for
Banach  spaces with separable dual.
  We will need the following lemma.
 \begin{lem}\label{convex means to zero yields weak null}
  Let $\ff$ be a regular thin family and
  $(y_s)_{s\in\ff}$ be an $\ff$-sequence in a Banach space $X$.
  Let  $L\in[\nn]^\infty$ and suppose that for every $\ee>0$ and $N\in[L]^\infty$  there exist
  $k\in\nn$, $\lambda_1,\ldots,\lambda_k>0$ and
  $(s_j)_{j=1}^k\in\text{{Plm}}(\ff\upharpoonright N)$ such
  that $\sum_{j=1}^k\lambda_j=1$ and $\|\sum_{j=1}^k\lambda_jy_{s_j}\|<\ee$.
Then for every $x^*\in X^*$, $\ee>0$ and $N\in[L]^\infty$ there
exists $M\in[N]^\infty$ such that
 $|x^*(y_s)|<\ee$ for every $s\in\ff\upharpoonright M$.
\end{lem}
\begin{proof}
  Let $x^*\in X^*$, $\ee>0$ and $N\in[L]^\infty$. By
  Proposition \ref{th2} there exists $M\in[N]^\infty$ such that exactly one
  of the following holds. (a) $|x^*(y_s)|<\ee$, for every $s\in\ff\upharpoonright
  M$, or (b) $x^*(y_s)\geq\ee$, for every $s\in\ff\upharpoonright
  M$, or (c) $x^*(y_s)\leq-\ee$, for every $s\in\ff\upharpoonright
  M$. It suffices to  show that  cases (b) and (c) cannot occur.  Indeed,
suppose  (b) holds true (the proof for case (c) is similar). By
our assumption there exist $k\in\nn$,
$\lambda_1,\ldots,\lambda_k>0$ and
  $(s_j)_{j=1}^k\in\text{{Plm}}(\ff\upharpoonright M)$ such
  that
$\sum_{j=1}^k\lambda_j=1$ and
    $\|\sum_{j=1}^k\lambda_jy_{s_j}\|<\ee$.
  But then \begin{equation}\ee>\Big\|\sum_{j=1}^k\lambda_jy_{s_j}\Big\|\geq x^*\Big(\sum_{j=1}^k\lambda_jy_{s_j}\Big)=
  \sum_{j=1}^k\lambda_jx^*(y_{s_j})\geq \ee,\end{equation} which is a
  contradiction.
\end{proof}

  \begin{cor}\label{singular in space with separable dual}
  Let $X$ be a Banach space  with separable dual. Let $\ff$ be a regular thin
  family,
  $(x_s)_{s\in\ff}$ be an $\ff$-sequence in  $X$ and $M\in[\nn]^\infty$
  such that $(x_s)_{s\in\ff\upharpoonright M}$ generates a
  singular $\ff$-spreading model $(e_n)_{n}$. Let $e_n=e'_n+e$ be the natural decomposition
  of  $(e_n)_{n}$. Then there
  exist $x\in X$  with $\|x\|=\|e\|_*$ and $N\in[M]^\infty$ such that setting
  $x'_s=x_s-x$,  $(x'_s)_{s\in\ff\upharpoonright N}$ is weakly null and admits $(e'_n)_{n}$ as
  a unique (up to isometry) $\ff$-spreading model.
\end{cor}
\begin{proof}
By Theorem \ref{singhighord}, there   exist $x\in X$ with
$\|x\|=\|e\|_*$ and $L\in[M]^\infty$ such that
$(x'_s)_{s\in\ff\upharpoonright L}$ admits $(e'_n)_{n}$ as a
unique $\ff$-spreading model. By applying Lemma \ref{convex means
to zero yields weak null} (for $x_s$ in place of $y_s$) and a
standard diagonalization for a countable dense subset of $X^*$ we
may choose $N\in[L]^\infty$ such that the $\ff$-subsequence
$(x'_s)_{s\in\ff\upharpoonright N}$ is in addition weakly null.
\end{proof}

\end{document}